\documentclass[a4paper,reqno, 12pt]{amsart}
\usepackage[utf8]{inputenc}
\usepackage[T1]{fontenc}

\usepackage{amsthm, thmtools}
\usepackage{stmaryrd}

\usepackage{amssymb, graphicx, enumerate, color}
\usepackage[usenames,dvipsnames,svgnames,table]{xcolor}
\usepackage[numbers,sort&compress]{natbib}

\usepackage[foot]{amsaddr}
\makeatletter
\let\old@setaddresses\@setaddresses
\def\@setaddresses{\bigskip\bgroup\parindent 0pt\let\scshape\relax\old@setaddresses\egroup}
\makeatother

\usepackage[unicode=true]{hyperref}
\hypersetup{
    pdftitle={A proof of Seymour's conjecture},
    pdfauthor={Tony Huynh, Gwena\"el Joret, Piotr Micek, David R. Wood},
    colorlinks,
    linkcolor={red!50!black},
    citecolor={blue!50!black},
    urlcolor={blue!80!black}
}

\usepackage[noabbrev,capitalise]{cleveref}

\newcommand{\doi}[1]{doi:\,\href{http://dx.doi.org/#1}{#1}}

\makeatletter
\def\thm@space@setup{
\thm@preskip=4mm
\thm@postskip=0mm
}
\makeatother

\let\leq\leqslant

\let\geq\geqslant

\let\subset\subseteq

\let\epsilon\varepsilon

\let\setminus\smallsetminus

\newcommand{\set}[1]{\{#1\}}

\newtheorem{theorem}{}[section]

\newtheorem{claim}[theorem]{}

\newtheorem{lemma}[theorem]{}

\newcommand{\OUT}{\mathsf{Out}}
\newcommand{\depth}{\mathsf{d}}
\newcommand{\height}{\mathsf{h}}

\newcommand{\ceil}[1]{\ensuremath{\protect\lceil #1\rceil}}

\usepackage{enumitem}
\renewenvironment{enumerate}{\begin{enumorig}[label=\textup{(\roman*)}, noitemsep, topsep=2pt plus 2pt, labelindent=.2em, leftmargin=*, widest=iii]}{\end{enumorig}}
\newenvironment{enumeratea}{\begin{enumorig}[label=\textup{(\alph*)}, noitemsep, topsep=2pt plus 2pt, labelindent=.2em, leftmargin=*, widest=iii]}{\end{enumorig}}

\newenvironment{enumeratearabic}{\begin{enumorig}[label=\textup{(\arabic*)}, noitemsep, topsep=2pt plus 2pt, labelindent=.2em, leftmargin=*, widest=iii]}{\end{enumorig}}

\renewenvironment{itemize}{\begin{itemorig}[label=\textbullet, noitemsep, topsep=2pt plus 2pt, labelindent=.5em, labelsep=.5em, leftmargin=*]}{\end{itemorig}}

\begin{document}

\title[A proof of Seymour's conjecture]{Seymour's conjecture on \\ $2$-connected graphs of large pathwidth}

\author[T.~Huynh]{Tony Huynh}
\address[T.~Huynh]{D\'epartement de Math\'ematique, Universit\'e Libre de Bruxelles, Brussels, Belgium}
\email{tony.bourbaki@gmail.com}

\author[G.~Joret]{Gwena\"el Joret}
\address[G.~Joret]{D\'epartement d'Informatique, Universit\'e Libre de Bruxelles, Brussels, Belgium}
\email{gjoret@ulb.ac.be}

\author[P.~Micek]{Piotr Micek}
\address[P.~Micek]{Theoretical Computer Science Department\\
  Faculty of Mathematics and Computer Science, Jagiellonian University, Krak\'ow, Poland}
\email{piotr.micek@tcs.uj.edu.pl}

\author[D.~R.~Wood]{David R. Wood}
\address[D.~R.~Wood]{School of Mathematics, Monash University, Melbourne, Australia}
\email{david.wood@monash.edu}

\thanks{G.\ Joret acknowledges support from the Australian Research Council and from an {\em Action de Recherche Concert\'ee} grant from the Wallonia-Brussels Federation in Belgium.
P.\ Micek is partially supported by a Polish National Science Center grant (SONATA BIS 5; UMO-2015/18/E/ST6/00299). T.\ Huynh is supported  by ERC Consolidator Grant 615640-ForEFront. 
T.\ Huynh, G.\ Joret, and P.\ Micek also acknowledge support from a joint grant funded by the Belgian National Fund for Scientific Research (F.R.S.--FNRS) and the Polish Academy of Sciences (PAN). 
Research of D.~R.~Wood is supported by the Australian Research Council.}

\begin{abstract}
We prove a conjecture of Seymour (1993) stating that for every apex-forest $H_1$ and outerplanar graph $H_2$ there is an integer $p$ such that every 2-connected graph of pathwidth at least $p$ contains $H_1$ or $H_2$ as a minor.  An independent proof was recently obtained by Dang and Thomas (\href{http://arxiv.org/abs/1712.04549}{arXiv:1712.04549}).
\end{abstract}

\maketitle

\section{Introduction}
\label{sec:intro}

Pathwidth is a graph parameter of fundamental importance, especially in graph structure theory. The \emph{pathwidth} of a graph $G$ is the minimum integer $k$ for which there is a sequence of sets $B_1,\dots,B_n\subseteq V(G)$ such that $|B_i| \leq k+1$ for each $i\in[n]$, for every vertex $v$ of $G$, the set $\{i\in[n]:v\in B_i\}$ is a non-empty interval, and for each edge $vw$ of $G$, some $B_i$ contains both $v$ and $w$.  

In the first paper of their graph minors series, Robertson and Seymour~\cite{RS83} proved the following theorem.

\begin{theorem} \label{noforest}
For every forest $F$, there exists a constant $p$ such that every graph with pathwidth at least $p$ contains $F$ as a minor.
\end{theorem}

The constant $p$ was later improved to $|V(F)|-1$ (which is best possible) by Bienstock, Robertson, Seymour, and Thomas~\cite{QuicklyForest}.  A simpler proof of this result was later found by Diestel~\cite{DiestelShort}.

Since forests have unbounded pathwidth, \ref{noforest} implies that a minor-closed class of graphs has unbounded pathwidth if and only if it includes all forests.  However, these certificates of large pathwidth are not 2-connected, so it is natural to ask for which minor-closed classes $\mathcal{C}$, does every \emph{2-connected} graph in $\mathcal{C}$ have bounded pathwidth? 

In 1993, Paul Seymour proposed the following answer (see \cite{Dean93}). A graph $H$ is an \emph{apex-forest} if $H-v$ is a forest for some $v \in V(H)$. A graph $H$ is \emph{outerplanar} if it has an embedding in the plane with all the vertices on the outerface. These classes are relevant since they both contain 2-connected graphs with arbitrarily large pathwidth. 
Seymour conjectured the following converse holds. 

\begin{theorem}
\label{SeymourConj}
For every apex-forest $H_1$ and outerplanar graph $H_2$ there is an integer $p$ such that every 2-connected graph of pathwidth at least $p$ contains $H_1$ or $H_2$ as a minor. 
\end{theorem}

Equivalently, \ref{SeymourConj} says that for a minor-closed class $\mathcal{C}$, every 2-connected graph in $\mathcal{C}$ has bounded pathwidth if and only if some apex-forest and some outerplanar graph are not in $\mathcal{C}$. 

The original motivation for conjecturing~\ref{SeymourConj} was to seek a version of~\ref{noforest} for matroids (see~\cite{DT17second}). Observe that apex-forests and outerplanar graphs are planar duals (see \ref{duality}). Since a matroid and its dual have the same pathwidth (see~\cite{Kashyap08} for the definition of matroid pathwidth),~\ref{SeymourConj} provides some evidence for a matroid version of~\ref{noforest}.

In this paper we prove~\ref{SeymourConj}.  An independent proof was recently obtained by Dang and Thomas~\cite{DT17second}. 

We actually prove a slightly different, but equivalent version of~\ref{SeymourConj}.  Namely, we prove that there are two unavoidable families of minors for $2$-connected graphs of large pathwidth.  We now describe our two unavoidable families.  

%A \emph{binary tree} is a rooted tree such that every vertex is either a leaf or has exactly two children. 
A \emph{binary tree} is a rooted tree such that every vertex has at most two children. 
For $\ell\geq 0$, the \emph{complete binary tree of height $\ell$}, denoted $\Gamma_\ell$, is the binary tree with $2^\ell$ leaves such that each root to leaf path has $\ell$ edges. It is well known that $\Gamma_\ell$ has pathwidth $\ceil{\ell/2}$. Let $\Gamma_{\ell}^{+}$ be the graph obtained from $\Gamma_\ell$ by adding a new vertex adjacent to all the leaves of $\Gamma_\ell$. See \cref{fig:cbts}.
Note that $\Gamma_{\ell}^{+}$ is a $2$-connected apex-forest, and its pathwidth grows as $\ell$ grows (since it contains $\Gamma_\ell$).

\begin{figure}[!ht]
\centering
\includegraphics{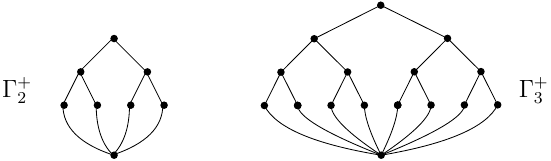}
\caption{Complete binary trees with an extra vertex adjacent to all the leaves.}
\label{fig:cbts}
\end{figure}

Our second set of unavoidable minors is defined recursively as follows.  Let $\nabla_1$ be a triangle with a \emph{root edge} $e$.  Let $H_1$ and $H_2$ be copies of $\nabla_\ell$ with root edges $e_1$ and $e_2$. Let $\nabla$ be a triangle with edges $e_1$, $e_2$ and $e_3$.  Define $\nabla_{\ell+1}$ by gluing each $H_i$ to $\nabla$ along $e_i$ and then declaring $e_3$ as the new root edge.  See \cref{fig:universal-outerplanars}.  Note that $\nabla_\ell$ is a $2$-connected outerplanar graph, and its pathwidth grows as $\ell$ grows (since it contains $\Gamma_{\ell-1}$).
\begin{figure}[ht]
\centering
\includegraphics{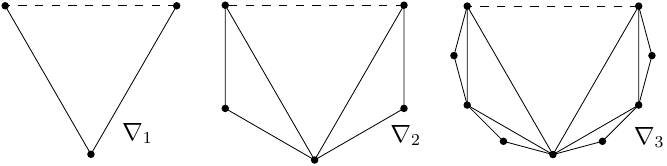}
\caption{\label{fig:universal-outerplanars}
Universal outerplanar graphs. The root edges are dashed.}
\end{figure}

The following is our main theorem.

\begin{theorem}\label{thm:main}
For every integer $\ell\geq 1$ there is an integer $p$ such that every $2$-connected graph of pathwidth at least $p$ contains $\Gamma_{\ell}^{+}$ or  $\nabla_\ell$ as a minor.
\end{theorem}

In Section~\ref{universal}, we prove that  every apex-forest is a minor of a sufficiently large $\Gamma_{\ell}^{+}$ and every outerplanar graph is a minor of a sufficiently large $\nabla_\ell$.  Thus, Theorem~\ref{thm:main} implies Seymour's conjecture. 

We actually prove the following theorem, which by~\ref{noforest}, implies~\ref{thm:main}.

\begin{theorem}\label{thm:main-CBT}
For all integers $\ell \geq 1$, there exists an integer $k$ such that  every $2$-connected graph $G$ with a $\Gamma_k$ minor contains $\Gamma_{\ell}^{+}$ or $\nabla_\ell$ as a minor.
\end{theorem}

Our approach is different from that of Dang and Thomas~\cite{DT17second}, who instead observe that by the Grid Minor Theorem~\cite{RS_five}, one may assume that $G$ has bounded treewidth but large pathwidth. Dang and Thomas then apply their machinery of `non-branching tree decompositions' to prove~\ref{SeymourConj}. 
 
The rest of the paper is organized as follows. \cref{universal} proves the universality of our two families.  In Sections~\ref{sec:beds} and~\ref{sec:binarypears}, we define `special' ear decompositions and prove that special ear decompositions always yield $\Gamma_{\ell}^{+}$ or $\nabla_\ell$ minors. In \cref{sec:findingbeds}, we prove that a minimal counterexample to~\ref{thm:main-CBT} always contains a special ear decomposition. Section~\ref{sec:theproof} concludes with short derivations of our main results.

%%%%%%%%%%%%%%%%%%%%%%
\section{Universality}
\label{universal}

This section proves some elementary (and possibly well-known) results. We include the proofs for completeness. 

\begin{lemma}
\label{duality}
Outerplanar graphs and apex-forests are planar duals.
\end{lemma}

\begin{proof}
Let $G$ be an apex-forest, where $G-v$ is a forest. 
Consider an arbitrary planar embedding of $G$. 
Note that every face of $G$ includes $v$ (otherwise $G-v$ would contain a cycle).
Let $G^*$ be the planar dual of $G$.
Let $f$ be the face of $G^*$ corresponding to $v$.
Since every face of $G$ includes $v$, every vertex of $G^*$ is on $f$.
So $G^*$ is outerplanar.

Conversely, let $G$ be an outerplanar graph.
Consider a planar embedding of $G$, in which every vertex is on the outerface $f$. 
Let $G^*$ be the planar dual of $G$. 
Let $v$ be the vertex of $G^*$ corresponding to $f$. 
If $G^*-v$ contained a cycle $C$, then a face of $G^*-v$ `inside' $C$ would correspond to a vertex of $G$ that is not on $f$. 
Thus $G^*-v$ is a forest, and $G^*$ is an apex-forest. 
\end{proof}

We now show that Theorem~\ref{thm:main} implies Seymour's conjecture, by proving two universality results.  

\begin{lemma}
\label{CBTuniversal}
Every apex-forest on $n\geq 2$ vertices is a minor of $\Gamma_{n-1}^{+}$. 
\end{lemma}

 If $H$ is a minor of $G$ and $v \in V(H)$, the \emph{branch set} of $v$ is the set of vertices of $G$ that are contracted to $v$.   \ref{CBTuniversal} is a corollary of the following.

%\begin{lemma}
%For every rooted tree $X$ with $n$ vertices, and for all $\ell\geq n$, $X$ is a minor of $\Gamma_\ell$, such that each branch set includes a leaf of $\Gamma_\ell$, and the root of $\Gamma_\ell$ is in the branch set corresponding to the root of $X$. 
%\end{lemma}
%
% \begin{proof} Let $r$ be the root of $X$. Let $v_1,\dots,v_d$ be the children of $r$ in $X$.  For $i\in[d]$, let $X_i$ be the subtree of $X$ rooted at $v_i$. Note that each $X_i$ has at most $n-d$ vertices. Let $T$ be a copy of $\Gamma_\ell$. We now construct the desired minor of $X$ in $T$. Let $P$ be the path from the root of $T$ to the leftmost leaf of $T$.  Let the branch set of $r$ be $P$. Say $P=(w_1,w_2,\dots,w_{\ell+1})$. For $i\in[d]$, let $x_i$ be the child of $w_i$ different from $w_{i+1}$. Let $T_i$ be the subtree of $T$ rooted at $x_i$. Note that $T_i$ is isomorphic to $\Gamma_{\ell-i}$. Since $\ell-i \geq \ell -d \geq n-d \geq |V(X_i)| $, by induction, $X_i$ is a rooted minor of $T_i$, such that each branch set includes a leaf of $T_i$, and the root of $T_i$ is in the branch set corresponding to $v_i$. By the latter property, there is an edge (namely, $w_ix_i$) between the branch set of $r$ and the branch set of $v_i$. Since the leaves of each $T_i$ are also leaves of $T$, we have the desired minor of $X$ in $T$.\end{proof} 

\begin{lemma}
Every tree with $n\geq 1$ vertices is a minor of $\Gamma_{n-1}$, such that each branch set includes a leaf of $\Gamma_{n-1}$.
\end{lemma}

\begin{proof} 
We proceed by induction on $n$. The base case $n=1$ is trivial. Let $T$ be a tree with $n\geq 2$ vertices. Let $v$ be a leaf of $T$. Let $w$ be the neighbour of $v$ in $T$. By induction, $T-v$ is a minor  of $\Gamma_{n-2}$, such that each branch set includes a leaf of $\Gamma_{n-2}$. In particular, the branch set for $w$ includes some leaf $x$ of $\Gamma_{n-2}$. Note that $\Gamma_{n-1}$ is obtained from $\Gamma_n$ by adding two new leaf vertices adjacent to each leaf of $\Gamma_{n-2}$. Let $y$ and $z$ be the leaf vertices of $\Gamma_{n-1}$ adjacent to $x$. 
Extend the branch set for $w$ to include $y$ and let $\{z\}$ be the branch set of $v$. For each leaf $u\neq x$ of $\Gamma_{n-2}$, if $u$ is in the branch set of some vertex of $T-v$, then extend this branch set to include one of the new leaves in $\Gamma_{n-1}$ adjacent to $u$. Now $T$ is a minor of $\Gamma_{n-1}$, such that each branch set includes a leaf of $\Gamma_{n-1}$.
\end{proof}

%Note that the dual of an apex-forest is an outerplanar graph (and vice versa). In particular, $\Gamma_n^+$ is dual to $\nabla_n$. NOT QUITE TRUE. 
%Deleting an edge in the dual corresponds to contracting the corresponding edge in the original graph (and contracting corresponds to deleting). 
%\ref{CBTuniversal} thus implies:

Our second universality result is for outerplanar graphs. 

\begin{lemma}
\label{OuterplanarUniversal}
Every outerplanar graph on $n\geq 2$ vertices is a minor of $\nabla_{n-1}$. 
\end{lemma}

% \comment{DW: 
% \emph{Attempted proof by duality}. 
% Let $G$ be an outerplanar graph on $n$ vertices. 
% It follows from Euler's formula that $G$ has at most $n-1$ faces. 
% By \ref{duality}, the dual $G^*$ is an apex-forest on at most $n-1$ vertices. 
% By \ref{CBTuniversal}, $G^*$ is a minor of $\Gamma_{n-2}^{+}$. 
% Deleting an edge in $G^*$ corresponds to contracting the corresponding edge in $G$, 
% and contracting an edge in $G^*$ corresponds to deleting the corresponding edge in $G$.
% Thus $G$ is a minor of the dual of $\Gamma_{n-2}^{+}$. 
% To finish the proof, we need that the dual of $\Gamma_{n-2}^{+}$ is a minor of $\nabla_{n'}$ for some $n'$. I don't see an easy argument why this is the case.} 

\ref{OuterplanarUniversal} is a corollary of the following.

\begin{lemma}
\label{OuterplanarUniversalInduction}
Every outerplanar triangulation $G$ on $n\geq 3$ vertices is a minor of $\nabla_{n-1}$, such that for every edge $vw$ on the outerface of $G$, there is a non-root edge on the outerface of $\nabla_{n-1}$ joining the branch sets of $v$ and $w$.
\end{lemma}

\begin{proof}
We proceed by induction on $n$. The base case, $G=K_3$, is easily handled as illustrated in \cref{OuterplanarBaseCase}. Let $G$ be an outerplanar triangulation with $n\geq 4$ vertices. Every such graph has a vertex $u$ of degree 2, such that if $\alpha$ and $\beta$ are the neighbours of $u$, then $G-u$ is an outerplanar triangulation and $\alpha\beta$ is an edge on the outerface of $G-u$. By induction, $G-u$ is a minor of $\nabla_{n-2}$, such that for every edge $vw$ on the outerface of $G-u$, there is a non-root edge $v'w'$ on the outerface of  $\nabla_{n-2}$ joining the branch sets of $v$ and $w$. In particular, there is a non-root edge $\alpha'\beta'$ of $\nabla_{n-2}$ joining the branch sets of $\alpha$ and $\beta$. Note that $\nabla_{n-1}$ is obtained from $\nabla_{n-2}$ by adding, for each non-root edge $pq$ on the outerface of $\nabla_{n-2}$, a new vertex adjacent to $p$ and $q$. Let the branch set of $u$ be the vertex $u'$ of $\nabla_{n-1}-V(\nabla_{n-2})$ adjacent to $\alpha'$ and $\beta'$. Thus $\nabla_{n-1}$ contains $G$ as a minor. Every edge on the outerface of $G$ is one of $u\alpha$ or $u\beta$, or is on the outerface of $G-u$. By construction, $u'\alpha'$ is a non-root edge on the outerface of $\nabla_{n-1}$ joining the branch sets of $u$ and $\alpha$. Similarly, $u'\beta'$ is a non-root edge on the outerface of $\nabla_{n-1}$ joining the branch sets of $u$ and $\beta$. For every edge $vw$ on the outerface of $G$, where $vw\not\in\{u\alpha,u\beta\}$, if $z$ is the vertex in $\nabla_{n-1}-V(\nabla_{n-2})$ adjacent to $v'$ and $w'$, extend the branch set of $v$ to include $z$. Now $zw'$ is an edge on the outerface of $\nabla_{n-1}$ joining the branch sets for $v$ and $w$. Thus for every edge $vw$ on the outerface of $G$, there is a non-root edge of $\nabla_{n-1}$ joining the branch sets of $v$ and $w$. 
\end{proof}

\begin{figure}[ht]
\centering
\includegraphics{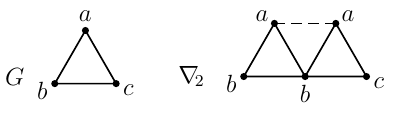}
\caption{\label{OuterplanarBaseCase}
Proof of \ref{OuterplanarUniversalInduction} in the base case. }
\end{figure}

%%%%%%%%%%%%%%%%%%%%%%
\section{Binary Ear Trees}
\label{sec:beds}

Henceforth, all graphs in this paper are finite and simple.  In particular, after contracting an edge, we suppress parallel edges and loops. Let $H$ and $G$ be graphs.  
We write $H \simeq G$ if $H$ and $G$ are isomorphic.  
Let $H \cup G$ be the graph with $V(H\cup G)=V(H) \cup V(G)$ and $E(H \cup G)=E(H) \cup E(G)$.   If $H$ is a subgraph of $G$, then an {\em $H$-ear} is a path in $G$ with its two ends in $V(H)$ but with no internal vertex in $V(H)$.  The \emph{length} of a path is its number of edges.  

For a vertex $v$ in a rooted tree $T$, let $T_v$ be the subtree of $T$ rooted at $v$. 
A vertex $v$ of $T$ is said to be {\em branching} if $v$ has at least two children. 

A \emph{binary ear tree} in a graph $G$ is a pair $(T, \mathcal{P})$, where $T$ is a binary tree, and $\mathcal{P}=\{P_x: x \in V(T)\}$ is a collection of paths in $G$ of length at least $2$ such that, for every non-root vertex $x$ of $T$ the following holds: 
\begin{enumerate}  
    \item $P_x$ is a $P_y$-ear, where $y$ is the parent of $x$ in $T$, and \label{item:parent}
    \item no internal vertex of $P_x$ is in $\bigcup_{z \in V(T) \setminus V(T_x)} V(P_z)$. \label{item:internally_disjoint}
\end{enumerate} 

A binary ear tree $(T, \mathcal{P})$ is {\em clean} if for every non-leaf vertex $y$ of $T$, there is an end of $P_y$ that is not contained in any $P_x$ where $x$ is a child of $y$. 

The main result of this section is the following.

\begin{theorem} \label{thm:binaryminors}
For every integer $\ell \geq 1$, if $G$ has a clean binary ear tree $(T, \mathcal{P})$ such that $T \simeq \Gamma_{3\ell-2}$, then $G$ contains $\Gamma_{\ell}^{+}$ or $\nabla_\ell$ as a minor.  
\end{theorem}

Before starting the proof, we first set up notation for a Ramsey-type result that we will need.  

If $p$ and $q$ are vertices of a tree $T$, then let $pTq$ denote the unique $pq$-path in $T$.
If $T'$ is a subdivision of a tree $T$, the vertices of $T'$ coming from $T$ are called \emph{original vertices} and the other vertices of $T'$ are called \emph{subdivision vertices}. Given a colouring of the vertices of $T=\Gamma_n$ with colours $\{\mathsf{red}, \mathsf{blue}\}$, we say that $T$ contains a \emph{red subdivision of $\Gamma_k$}, if it contains a subdivision $T'$ of $\Gamma_k$ such that all the original vertices of $T'$ are red, and for all $a,b \in V(T')$ with $b$ a descendant of $a$, the path $aTb$ is descending. (Here a path is \emph{descending} if it is contained in a path that starts at the root.)\ 
Define $R(k,\ell)$ to be the minimum integer $n$ such that every colouring of $\Gamma_n$ with colours $\{\mathsf{red}, \mathsf{blue}\}$  contains a red subdivision of $\Gamma_k$ or a blue subdivision of $\Gamma_\ell$.  We will use the following easy result.

\begin{lemma} \label{lem:ramseytree}
$R(k, \ell) \leq k+\ell$ for all integers $k,\ell \geq 0$.
\end{lemma}

\begin{proof}
We proceed by induction on $k+\ell$. As base cases, it is clear that $R(k, 0)=k$ and $R(0,\ell)=\ell$ for all $k, \ell$.
For the inductive step, assume $k,\ell \geq 1$ and let $T$ be a $\{\mathsf{red}, \mathsf{blue}\}$-coloured copy of $\Gamma_{k+\ell}$. By symmetry, we may assume that the root $r$ of $T$ is coloured red.
Let $T_1$ and $T_2$ be the components of $T-r$, both of which are copies of $\Gamma_{k+\ell-1}$. If $T_1$ or $T_2$ contains a blue subdivision of $\Gamma_\ell$, then so does $T$ and we are done. 
By induction, $R(k-1,\ell) \leq k-1+\ell$, so both $T_1$ and $T_2$ contain a red subdivision of $\Gamma_{k-1}$. Add the paths from $r$ to the roots of these red subdivisions. We obtain a red subdivision of $\Gamma_k$, as desired. 
\end{proof}

The following observation will be helpful when considering subdivision vertices. 

\begin{claim}
\label{claim:subdivision_vertices}
Let $G$ be a graph having a clean binary ear tree $(T, \mathcal{P})$ with $\mathcal{P}=\{P_v: v \in V(T)\}$. 
Suppose that $y$ is a degree-$2$ vertex in $T$ with parent $x$ and child $z$. 
Then there is a clean binary ear tree $(T/yz, \mathcal{P'})$ of $G$, with $\mathcal{P'}=\{P'_v: v \in V(T/yz)\}$ where $P'_v = P_v$ for all $v\in V(T) \setminus \{y, z\}$, and $P'_{yz}$ is the unique $P_x$-ear contained in $P_y \cup P_z$ that contains $P_z$, where the vertex resulting from the contraction of edge $yz$ is denoted $yz$ as well.  
\end{claim}
\begin{proof}
Property~\ref{item:parent} of the definition of binary ear trees holds for vertex $yz$ of $T/yz$ by our choice of $P'_{yz}$. 
Property~\ref{item:internally_disjoint} holds for $yz$ because it held for $y$ and for $z$ in $(T, \mathcal{P})$. 
Also, these two properties hold for children of $yz$ in $T/yz$ (if any) because they held for $z$ before. 
Thus, $(T/yz, \mathcal{P'})$ is a binary ear tree. 
Finally, note that cleanliness of the binary ear tree $(T/yz, \mathcal{P'})$ follows from that of $(T, \mathcal{P})$, and the fact that the ends of $P'_{yz}$ are the same as the ones of $P_{y}$. 
\end{proof}

We now prove~\ref{thm:binaryminors}.

\begin{proof}[Proof of~\ref{thm:binaryminors}]
Let $t$ be a non-leaf vertex of $T$. Let $u$ and $v$ be the children of $t$. Let $u_1$ and $u_2$ be the ends of $P_{u}$. Let $v_1$ and $v_2$ be the ends of $P_{v}$. We say that $t$ is \emph{nested} if $u_1P_tu_2 \subseteq v_1P_tv_2$ or $v_1P_tv_2 \subseteq u_1P_tu_2$.  If $t$ is not nested, then $t$ is \emph{split}. See~Figures~\ref{fig:nested} and~\ref{fig:split}. 
Regarding $\mathsf{split}$ and $\mathsf{nested}$ as colours, we apply~\ref{lem:ramseytree} to the tree $T$ with the leaves removed, and obtain a tree $T^*$ which is a $\mathsf{split}$ subdivision of $\Gamma_{\ell-1}$ or a $\mathsf{nested}$ subdivision of $\Gamma_{2\ell-2}$. 
For each leaf of $T^*$, add back its two children in $T$. 
This way, we deduce that $T$ contains either a subdivision of $\Gamma_{\ell}$ with all branching vertices split, or a subdivision of $\Gamma_{2\ell-1}$ with all branching vertices nested.  
In the first case, we will find a $\nabla_\ell$ minor, while in the second we will find a $\Gamma_{\ell}^+$ minor. 
The two cases are covered by~\ref{claim:splitT1} and~\ref{claim:nestedT2}.
 
\begin{figure}[!ht]
\centering
\includegraphics{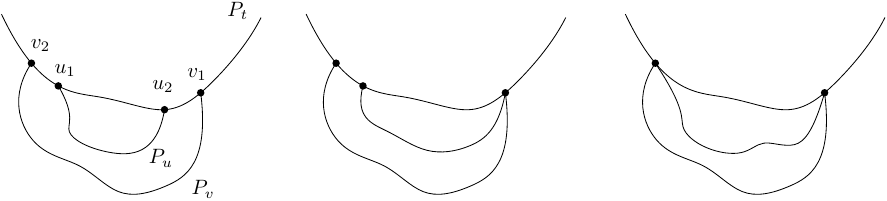}
\caption{Examples of a nested vertex $t$ with a path $P_t$ in a clean binary ear tree.}
\label{fig:nested}
\end{figure}
\begin{figure}[!ht]
\centering
\includegraphics{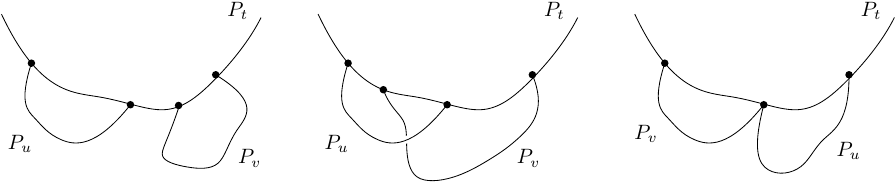}
\caption{Examples of a split vertex $t$ with a path $P_t$ in a clean binary ear tree.}
\label{fig:split}
\end{figure}

\begin{claim}
\label{claim:splitT1}
If $T$ contains a subdivision $T^{1}$ of $\Gamma_{\ell}$ such that every branching vertex is split, 
then $\bigcup_{t \in V(T^{1})} P_t$ contains $\nabla_\ell$ as a minor.
\end{claim}

\begin{proof}[Subproof]
Consider the clean binary ear tree `induced by' the subtree $T^{1}$, that is, the pair $(T^1, \mathcal{P}^1)$ where $\mathcal{P}^1=\{P_t: t \in V(T^1)\}$. 
First, for every subdivision vertex $y$ of $T^{1}$ with child $z$, we apply~\ref{claim:subdivision_vertices} to $(T^1, \mathcal{P}^1)$ in order to suppress vertex $y$. 
Note that every branching vertex of $T^{1}$ stays split. 
In particular, this is true if $z$ is branching. 
Hence, we may assume from now on that $T^{1}$ has no subdivision vertices. 

Let $P$ be a path in a graph $G$.  
Let $\nabla_{\ell}^-$ be the graph obtained from $\nabla_{\ell}$ by deleting its root edge $xy$.  
We say that a $\nabla_{\ell}^-$ minor in $G$ is \emph{rooted on $P$} if the two roots of the $\nabla_{\ell}^-$ minor are the ends of $P$. 
(By `roots' we mean the ends of the root edge.)  

We prove the following technical statement.  
Let $m \geq 0$ be an integer, and let $T'$ be a subtree of $T^1$ isomorphic to $\Gamma_{m}$ such that all branching vertices of $T'$ are split, then $\bigcup_{t \in V(T')} P_t$ contains a $\nabla_{m+1}^-$ minor rooted on $P_r$, where $r$ is the root of $T'$.  

This proves~\ref{claim:splitT1} for $\ell \geq 2$, since $\nabla_{\ell+1}^-$ contains a $\nabla_\ell$ minor.   
For $\ell=1$, \ref{claim:splitT1} is straightforward. 

We prove the above technical statement by induction on $m$. 
The case $m=0$ is clear since then $T'$ is a single vertex $v$ and $\nabla_1^-$ is just a path with three vertices. 
(Here we use that $|V(P_v)| \geq 3$.) 

For the inductive step, let $a$ and $b$ be the children of $r$. 
By induction, $G_a:=\bigcup_{t \in V(T'_{a})} P_t$ contains a $\nabla_{m}^-$ minor $H_a$ rooted on $P_{a}$, and $G_b:=\bigcup_{t \in V(T'_{b})} P_t$ contains a $\nabla_{m}^-$ minor $H_b$ rooted on $P_{b}$. 

We prove that $G_a$ and $G_b$ are vertex-disjoint, except possibly at a vertex of $V(P_a) \cap V(P_b)$ (there is at most one such vertex since $r$ is split).  
Suppose $v$ is a vertex appearing in both $G_a$ and $G_b$. 
Let $x$ be the vertex in $T'_a$ closest to the root such that $v\in V(P_x)$
and let $y$ be the vertex in $T'_b$ closest to the root such that $v\in V(P_y)$.
By property~\ref{item:internally_disjoint} of binary ear trees we know that no internal vertex of $P_x$ lies in
$\bigcup_{z \in V(T^1) \setminus V(T'_x)} V(P_z)$.
Since $y \in V(T^1) \setminus V(T'_x)$ and $v\in V(P_y)$, we conclude that $v$ is an end of $P_x$.
This means that $v$ lies in $T'_p$ where $p$ is the parent of $x$ in $T'$.
By the choice of $x$ this is only possible when $x=a$.
Thus, $v$ is an end of $P_a$ and lies in $P_r$.
By a symmetric argument we conclude that $v$ is an end of $P_b$ as well, as desired. 

Let $a_1$ and $a_2$ be the ends of $P_a$, $b_1$ and $b_2$ be the ends of $P_b$, and $r_1$ and $r_2$ be the ends of $P_r$. 
By symmetry, we may assume that the ordering of these points along $P_r$ is either  $r_1, a_1, b_1, a_2, b_2, r_2$ or $r_1, a_1, a_2, b_1, b_2, r_2$. (Note that some vertices may coincide.)   
Using the observation from the previous paragraph, we obtain a $\nabla_{m+1}^-$ minor rooted on $P_r$ by considering the union of the $\nabla_{m}^-$ minor rooted on $P_a$ and the $\nabla_{m}^-$ minor rooted on $P_b$ that we were given, and contracting the following three subpaths of $P_r$: $r_1P_ra_1$, $a_2P_rb_1$, and $b_2P_rr_2$. 
Notice that if $G_a$ and $G_b$ have a vertex $v$ in common, then $v=a_2=b_1$.  
See Figure~\ref{fig:getting-outerplanar} for an illustration of the construction.  
\end{proof}

\begin{figure}[!ht]
\centering
\includegraphics{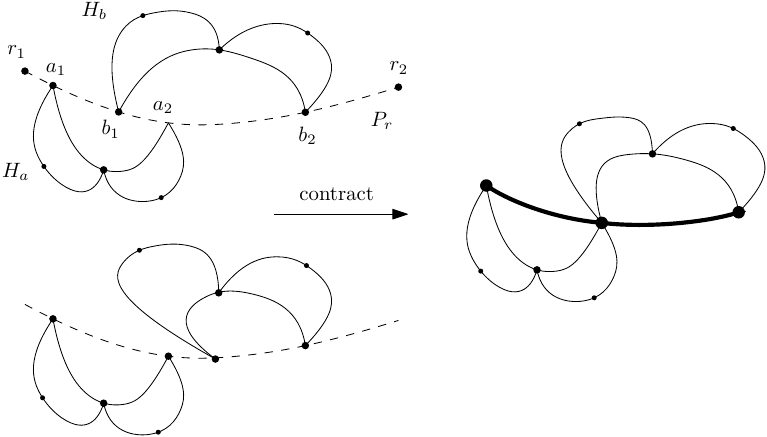}
\caption{Inductively constructing a $\nabla_3^-$ minor.}
\label{fig:getting-outerplanar}
\end{figure}

\begin{claim}
\label{claim:nestedT2}
If $T$ contains a subdivision $T^{2}$ of $\Gamma_{2\ell-1}$ such that every branching vertex is nested, 	
then $\bigcup_{t \in V(T^{2})} P_t$ contains $\Gamma^+_\ell$ as a minor.
\end{claim}

\begin{proof}[Subproof]
Consider the clean binary ear tree $(T^2, \mathcal{P}^2)$ where $\mathcal{P}^2=\{P_t: t \in V(T^2)\}$. 
First, for every subdivision vertex $y$ of $T^{2}$ with child $z$, we apply~\ref{claim:subdivision_vertices} to $(T^2, \mathcal{P}^2)$ in order to suppress vertex $y$. 
Note that every branching vertex of $T^{2}$ stays nested. 
In particular, this is true if $z$ is branching. 
Hence, we may assume from now on that $T^{2}$ has no subdivision vertices.

Orient each path in $\mathcal{P}^2$ inductively as follows.  
Let $r$ be the root of $T^2$ and orient $P_r$ arbitrarily. 
If $P_s$ has already been oriented and $t$ is a child of $s$ in $T^2$, then orient $P_t$ so that $P_s \cup P_t$ does not contain a directed cycle.  
Consider each path in $\mathcal{P}^2$ to be oriented from left to right, and thus with left and right ends. 

Let $t$ be a non-leaf vertex of $T^{2}$ and let $u$ and $v$ be the children of $t$.
Define $t$ to be \emph{left-good} if the left end of $P_t$ is not in $P_u$ nor $P_v$.
Define $t$ to be \emph{right-good} if the right end of $P_t$ is not in $P_u$ nor $P_v$.
Since $(T^2,\mathcal{P}^2)$ is clean we know that every non-leaf vertex $t$ of $T^{2}$ is left-good or right-good.
We colour the non-leaf vertices of $T^2$ with $\mathsf{left}$ and $\mathsf{right}$ in such a way that when a vertex is coloured $\mathsf{left}$ ($\mathsf{right}$), then it is left-good (right-good).  
Applying \ref{lem:ramseytree} on the tree $T^2$ with branching vertices coloured this way in which we remove all the leaves, we obtain a subdivision $T^*$ of $\Gamma_{\ell-1}$ such that all original vertices are coloured $\mathsf{left}$, or all are coloured $\mathsf{right}$, say without loss of generality $\mathsf{left}$. 
For every leaf of $T^*$, add back to $T^*$ its two children in $T^2$, and denote by $T^3$ the resulting tree. 
Note that $T^3$ is a subdivision of $\Gamma_{\ell}$ and all branching vertices of $T^3$ are left-good.  

We focus on the clean binary ear tree $(T^3, \mathcal{P}^3)$ induced by $T^3$, where $\mathcal{P}^3=\{P_t: t \in V(T^3)\}$. 
Then, for every subdivision vertex $y$ of $T^{3}$ with child $z$, we apply~\ref{claim:subdivision_vertices} to $(T^3, \mathcal{P}^3)$ in order to suppress vertex $y$, as before.  
Note that every branching vertex of $T^{3}$ stays nested and left-good. 
Hence, we may assume from now on that $T^{3}$ has no subdivision vertices. 

Let $t$ be a non-leaf vertex of $T^{3}$ and $u$ and $v$ be the children of $t$ in $T^3$. Let $f(t)$ be the first vertex of $P_t$ that is a left end of either $P_{u}$ or of $P_{v}$.  
Note that $f(t)$ is not the left end of $P_t$, since $t$ is left-good. 
Let $e(t)$ be the last edge of $P_t$ incident to a left end of either $P_{u}$ or $P_{v}$.  
If $t$ is a leaf of $T^3$, we define $f(t)$ to be any internal vertex of $P_t$ and $e(t)$ to be the last edge of $P_t$ incident to $f(t)$.   

Let $H:=\bigcup_{t \in V(T^{3})}P_t$ and $M:= \{e(t): t \in V(T^{3})\}$. 
Since every branching vertex of $T^3$ is nested, $H \backslash M$ contains two components $H_{\text{left}}$ and $H_{\text{right}}$ such that $H_{\text{left}}$ contains all left ends of $\{P_t : t \in V(T^{3})\}$ and $H_{\text{right}}$ contains all right ends of $\{P_t : t \in V(T^{3})\}$. 
Using that every branching vertex of $T^3$ is left-good, it is easy to see that  $H_{\text{left}}$ contains a subdivision $T^{4}$ of $\Gamma_\ell$ whose set of original vertices is $\{f(t): t \in V(T^{3})\}$; see Figure~\ref{fig:getting-tree-with-apex}. By construction, each leaf of $T^{4}$ is incident to an edge in $M$. Also, $H_{\text{right}}$ is clearly connected. Therefore, after contracting all edges of $H_{\text{right}}$, $T^4 \cup M \cup H_{\text{right}}$ contains a $\Gamma_{\ell}^{+}$ minor.
\end{proof}
This ends the proof of~\ref{thm:binaryminors}.
\end{proof}

\begin{figure}[!ht]
\centering
\includegraphics{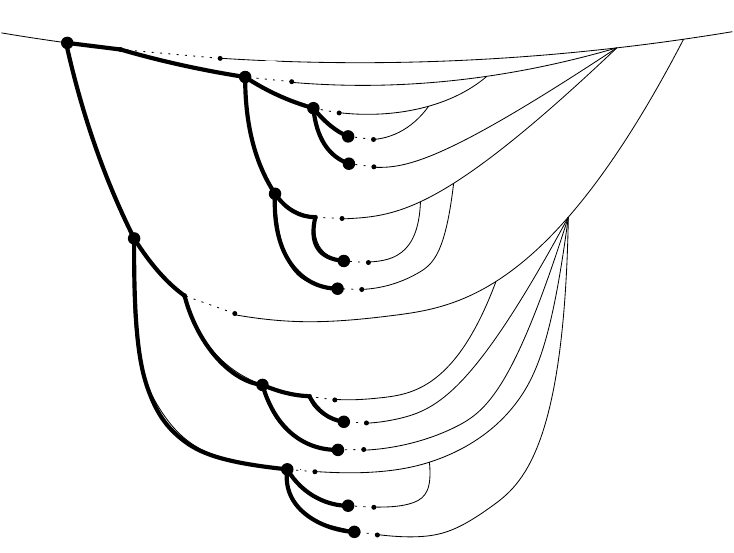}
\caption{A $\Gamma_3$ minor in $H_{\text{left}}$.}
\label{fig:getting-tree-with-apex}
\end{figure}

\section{Binary Pear Trees} \label{sec:binarypears}

In order to prove our main theorem, we need something slightly more general than binary ear trees, which we now define.
A \emph{binary pear tree} in a graph $G$ is a pair $(T, \mathcal{B})$, where $T$ is a binary tree, and $\mathcal{B}=\{(P_x, Q_x): x \in V(T)\}$ is a collection of pairs of paths of $G$ of length at least $2$ such that $P_x \subseteq Q_x$ for all $x \in V(T)$, and the following properties are satisfied for each non-root vertex $x \in V(T)$. 
\begin{enumerate}
    \item $Q_x$ is a $P_y$-ear, where $y$ is the parent of $x$ in $T$;  \label{def:bpt_parent}   
    \item if $x$ has no sibling then no internal vertex of $Q_x$ is in $\bigcup_{z \in V(T) \setminus V(T_x)} V(Q_z)$; \label{def:bpt_no_sibling} 
    \item if $x$ has a sibling $x'$ then \label{def:bpt_sibling}
    \begin{itemize}
    	\item no internal vertex of $Q_x$ is in $\bigcup_{z \in V(T) \setminus (V(T_x) \cup V(T_{x'}))} V(Q_z)$, and
    	\item no internal vertex of $P_x$ is in $Q_{x'}$. 
    \end{itemize}
\end{enumerate}
Furthermore, the binary pear tree is {\em clean} if for every non-leaf vertex $y$ of $T$, there is an end of $P_y$ that is not contained in any $Q_x$ where $x$ is a child of $y$.

Note that if $(T, \{P_x : x \in V(T)\})$ is a clean binary ear tree, then $(T, \{(P_x, P_x): x \in V(T)\})$ is a clean binary pear tree.  We now prove the following converse. 

\begin{theorem} \label{thm:pears make ears}
If $G$ has a clean binary pear tree $(T, \mathcal{B})$, then $G$ has a minor $H$ such that $H$ has a clean binary ear tree $(T, \mathcal{P})$.
\end{theorem}

\begin{proof}
Say $\mathcal{B}=\{(P_v, Q_v): v \in V(T)\}$.  We prove the stronger result that there exist $H$ and $(T, \{P_v': v \in V(T)\})$ such that $H$ is a minor of $G$, $(T, \{P_v': v \in V(T)\})$ is a clean binary ear tree in $H$, and $P_v \subseteq P'_v$ for all leaves $v$ of $T$. 
This last property will be referred to as the {\em leaf property}; note that this is a property of $(T, \{P_v': v \in V(T)\})$ w.r.t.\ the pair $(T, \mathcal{B})$ (which is fixed).  
Arguing by contradiction, suppose that this result is not true. Among all counterexamples, choose $(G, (T, \mathcal{B}))$ such that $|E(G)|$ is minimum.  This clearly implies that $|V(T)|>1$.  

Let $y$ be a deepest leaf in $T$. 
If $y$ has a sibling, let $z$ denote this sibling, which is also a leaf of $T$. 
Let $x$ be the parent of $y$ in $T$. 
Delete from $G$ the internal vertices of $Q_y$ and $Q_z$ (if $z$ exists), and denote by $G^-$ the resulting graph. 
Note that $|E(G^-)| < |E(G)|$ since $Q_y$ has length at least $2$.  
Let $T^-$ be the tree obtained from $T$ by removing $y$ and $z$ (if $z$ exists). 
Notice that no internal vertex of $Q_y$ or $Q_z$ appears in a path $Q_v$ with $v \in V(T^-)$, by properties~\ref{def:bpt_no_sibling} and~\ref{def:bpt_sibling} of the definition of binary pear trees. 
Thus $(T^-, \{(P_v, Q_v): v \in V(T^-)\})$ is a clean binary pear tree. 
By minimality, $G^-$ has a minor $H^-$ such that $H^-$ has a clean binary ear tree $(T^-, \{P_v^-: v \in V(T^-)\})$ 
such that $P_v \subseteq P_v^-$ for all leaves $v$ of $T^-$.  
Since $x$ is a leaf of $T^-$, we have $P_x \subseteq P_x^-$. 

Notice that $Q_y$ and $Q_z$ (if $z$ exists) are $P_x^-$-ears. 
If $z$ does not exist, then let $P_y^- := Q_y$ and observe that $(T, \{P_v^-: v \in V(T)\})$ is a clean binary ear tree satisfying the leaf property, contradicting the fact that $(G, (T, \mathcal{B}))$ is a counterexample. 
Thus, $z$ must exist. 

Consider an internal vertex $v$ of $Q_y$. 
If $v$ is included in $Q_z$ then $v$ cannot be an end of $Q_z$, because ends of $Q_z$ are in $P_x$, which would imply that $v$ is an end of $Q_y$ as well. 
Thus, if $Q_y$ and $Q_z$ have a vertex in common, either this vertex is a common end of both paths, or it is internal to both paths. 

If $Q_y$ and $Q_z$ have no internal vertex in common, let $P_y^- := Q_y$ and $P_z^- := Q_z$. 
Note that $(T, \{P_v^-: v \in V(T)\})$ is a clean binary ear tree satisfying the leaf property, a contradiction. 
Hence, $Q_y$ and $Q_z$ must have at least one internal vertex in common. 

Next, given an edge $e \in E(G)$ and a path $P$ in $G$, define $P \sslash e$ to be $P$ if $e \notin E(P)$ and $P / e$ if $e \in E(P)$, and let  $\mathcal{B} / e:=\{(P_v \sslash e, Q_v \sslash e): v \in V(T)\}$. 
Suppose that there is an edge $e \in E(Q_y) \cap E(Q_z)$. Since $|E(P_y)|\geq 2$ and $|E(P_z)|\geq 2$, property~\ref{def:bpt_sibling} of the definition of binary pear trees implies that  $e \notin E(P_y) \cup E(P_z)$. 
Thus $P_y \sslash e=P_y$ and $P_z \sslash e=P_z$.  
It follows that $(T, \mathcal{B} / e)$ is a clean binary pear tree of $G / e$, which contradicts the minimality of the counterexample. 
Hence, no such edge $e$ exists. 

So far we established that the two paths $Q_y$ and $Q_z$ have at least one internal vertex in common and are edge-disjoint. 
The rest of the proof is split into a number of cases. 
In each case, we show that either there is an edge $e$ of $G$ such that $G \backslash e$ still has a clean binary pear tree which is indexed by the same tree $T$, or that there is a way to modify $(T, \mathcal{B})$ so that it remains a clean binary pear tree of $G$, and after the modification the two paths $Q_y$ and $Q_z$ have at least one edge in common. 
Note that each outcome contradicts the minimality of our counterexample; in the latter case, this is because we can then apply the argument of the previous paragraph and obtain a smaller counterexample. 

Let us now proceed with the case analysis, see Figure~\ref{fig:pear-to-ear} for an illustration of the different cases.  
Choose an orientation of $P_x$ from left to right, let $x_1$ denote its left end and $x_2$ denote its right end, and let $y_1, y_2$ and $z_1, z_2$ be the two ends of respectively $Q_y$ and $Q_z$ on $P_x$, ordered from left to right. 
Given two vertices $u,v$ of $P_x$, let us simply write $u \leq v$ if $u=v$ or $u$ is to the left of $v$ on $P_x$. 
Without loss of generality, we may assume that $y_1 \leq z_1$. 

Recalling that $Q_y$ and $Q_z$ have an internal vertex in common, let $v_1$ be the first such vertex on the path $Q_y$ starting from $y_1$. 
Note that either $P_y \subseteq y_1Q_yv_1$ or $P_y \subseteq v_1Q_yy_2$, and similarly either $P_z \subseteq z_1Q_zv_1$ or $P_z \subseteq v_1Q_zz_2$, by property~\ref{def:bpt_sibling} of the definition of binary pear trees.  

\begin{figure}[!t]
\centering
\includegraphics{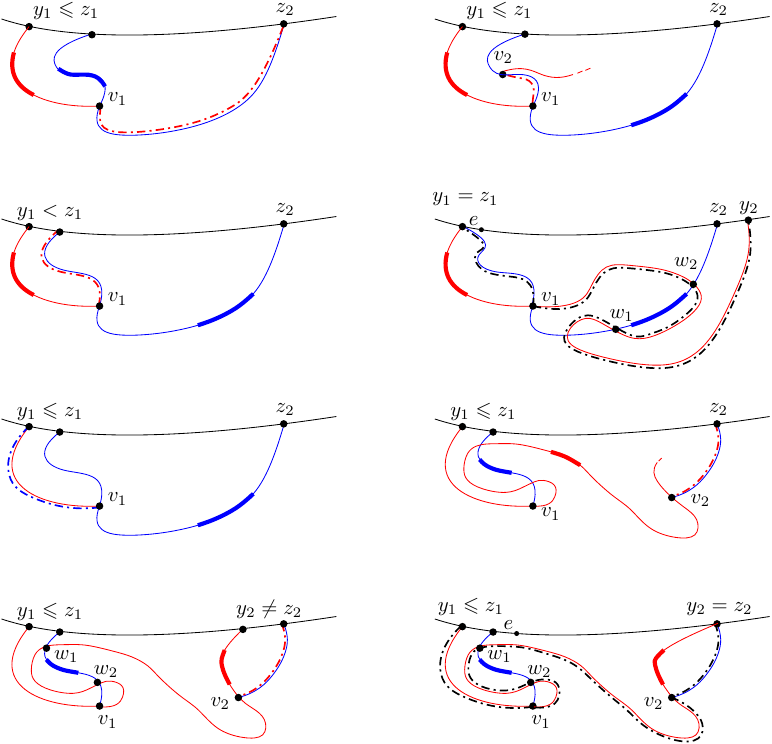}
\caption{Cases in the proof of~\ref{thm:pears make ears}. $P_x$ is drawn in black, $Q_y$ in red, and $Q_z$ in blue. The bold subpaths of $Q_y$ and $Q_z$ denote respectively $P_y$ and $P_z$. The dotted lines illustrate the modifications of the paths $P_x, Q_y, Q_z$.}
\label{fig:pear-to-ear}
\end{figure}

First suppose that $P_y \subseteq y_1Q_yv_1$ and $P_z \subseteq z_1Q_zv_1$. 
Let $Q^1_y := y_1Q_yv_1Q_zz_2$. 
(The superscript denotes the case number.) 
It is easily checked that replacing $Q_y$ with $Q^1_y$ in $(T, \mathcal{B})$ gives another clean binary pear tree of $G$. 
Moreover, $Q^1_y$ and $Q_z$ have the path $v_1Q_zz_2$ in common, which contains at least one edge, as desired. 

Next suppose that $P_y \subseteq y_1Q_yv_1$ and $P_z \subseteq v_1Q_zz_2$.   
We consider whether some internal vertex of the path $v_1Q_zz_1$ is in $Q_y$. 
If there is one, let $v_2$ be the last such vertex that is met when going along $Q_y$ from $y_1$ to $y_2$. 
Let $Q^2_y := y_1Q_yv_1Q_zv_2Q_yy_2$, and replace $Q_y$ with $Q^2_y$ in $(T, \mathcal{B})$ as in the previous paragraph. 
Note that $Q^2_y$ and $Q_z$ have the path $v_1Q_zv_2$ in common, and thus at least one edge in common, as desired.  

If no internal vertex of $v_1Q_zz_1$ is in $Q_y$, we consider whether $y_1 < z_1$ or $y_1 = z_1$. 
If $y_1 < z_1$, let $Q^3_y := y_1Q_yv_1Q_zz_1$, and replace $Q_y$ with $Q^3_y$ in $(T, \mathcal{B})$. 
In particular, $Q^3_y$ and $Q_z$ now have the path $v_1Q_zz_1$ in common, and thus at least one edge in common, as desired.  

If $y_1 = z_1$, we adopt a different strategy. 
Let $P^4_x:=x_1P_xy_1Q_zv_1Q_yy_2P_xx_2$ and let $Q^4_x$ be the path obtained from $Q_x$ by replacing the $P_x$ section with $P^4_x$. 
Let $Q^4_y := y_1Q_yv_1$. 
Let $w_1$ be the first vertex of $Q_y$ that is met when starting in $P_z$ and walking along $Q_z$ toward $z_1$. 
(Note that possibly $w_1 = v_1$.) 
Let $w_2$ be the first vertex of $Q_y$ that is met when starting in $P_z$ and walking along $Q_z$ toward $z_2$, if there is one.  
Let $Q^4_z := w_1Q_zw_2$ if $w_2$ exists, otherwise let $Q^4_z := w_1Q_zz_2P_xy_2$. 
Finally, let $e$ be the edge of $P_x$ incident to $z_1$ that is to the right of $z_1$. 
Observe that $e$ is not included in any of the three paths $Q^4_x, Q^4_y, Q^4_z$.  
Now, it can be checked that replacing $P_x, Q_x, Q_y, Q_z$ in $(T, \mathcal{B})$ with their newly defined counterparts produces a clean binary pear tree of $G \backslash e$, giving the desired contradiction. 
This concludes the case that $P_y \subseteq y_1Q_yv_1$ and $P_z \subseteq v_1Q_zz_2$.    

Next suppose that $P_y \subseteq v_1Q_yy_2$ and $P_z \subseteq v_1Q_zz_2$.   
Let $Q^5_z := y_1Q_yv_1Q_zz_2$. 
Replacing $Q_z$ with $Q^5_z$ in $(T, \mathcal{B})$ gives another clean binary pear tree of $G$. 
Moreover, $Q_y$ and $Q^5_z$ have the path $y_1Q_yv_1$ in common, which contains at least one edge, as desired. 

Finally, suppose that $P_y \subseteq v_1Q_yy_2$ and $P_z \subseteq z_1Q_zv_1$. 
Let $v_2$ be the first common internal vertex of $Q_y$ and $Q_z$ that is met when starting in $z_2$ and walking along $Q_z$ toward $v_1$. 
(Note that possibly $v_2=v_1$.) 
If $P_y \subseteq v_1Q_yv_2$ then let $Q^6_y := y_1Q_yv_2Q_zz_2$.  
Replacing $Q_y$ with $Q^6_y$ in $(T, \mathcal{B})$ gives another clean binary pear tree of $G$. 
Moreover, $Q^6_y$ and $Q_z$ have the path $v_2Q_zz_2$ in common, which contains at least one edge, as desired. 

If $P_y \subseteq v_2Q_yy_2$ then consider whether $y_2 = z_2$. 
If $y_2 \neq z_2$ then let $Q^7_y := y_2Q_yv_2Q_zz_2$.  
Replacing $Q_y$ with $Q^7_y$ in $(T, \mathcal{B})$ gives another clean binary pear tree of $G$. 
Moreover, $Q^7_y$ and $Q_z$ have the path $v_2Q_zz_2$ in common, which contains at least one edge, as desired. 

If $y_2 = z_2$, then let $P^8_x:=x_1P_xy_1Q_yv_2Q_zz_2P_xx_2$ and let $Q^8_x$ be the path obtained from $Q_x$ by replacing the $P_x$ section with $P^8_x$. 
Let $Q^8_y := v_2Q_yy_2$. 
Let $w_1$ be the first vertex of $Q_y$ that is met when starting in $P_z$ and walking along $Q_z$ toward $z_1$, if there is one.  
Let $w_2$ be the first vertex of $Q_y$ that is met when starting in $P_z$ and walking along $Q_z$ toward $z_2$. 
(Note that possibly $w_2 = v_1$.) 
Let $Q^8_z := w_1Q_zw_2$ if $w_1$ exists, otherwise let $Q^8_z := y_1P_xz_1Q_zw_2$.  
Let $e$ be the edge of $P_x$ incident to $z_1$ that is to the right of $z_1$. 
Observe that $e$ is not included in any of the three paths $Q^8_x, Q^8_y, Q^8_z$.  
Now, it can be checked that replacing $P_x, Q_x, Q_y, Q_z$ in $(T, \mathcal{B})$ with their newly defined counterparts produces a clean binary pear tree of $G \backslash e$, giving the desired contradiction. 
This concludes the proof. 
\end{proof}

\section{Finding Binary Pear Trees}
\label{sec:findingbeds}

A binary tree is {\em full} if every internal vertex has exactly two children. 
The main result of this section is the following.  

\begin{theorem} \label{lem:findingbeds}
For all integers $\ell \geq 1$ and $k\geq 9\ell^2 - 3\ell +1$, if $G$ is a minor-minimal $2$-connected graph containing a subdivision of $\Gamma_k$ and $T^1$ is a full binary tree of height at most $3\ell-2$, then either $G$ contains $\Gamma_{\ell}^{+}$ as a minor, or $G$ contains a clean binary pear tree $(T^1, \mathcal{B})$.
\end{theorem}

We proceed via a sequence of lemmas.

\begin{lemma} \label{lem:spanning}
If $G$ is a minor-minimal $2$-connected graph containing a subdivision of $\Gamma_k$, then every subdivision of $\Gamma_k$ in $G$ is a spanning tree.
\end{lemma}

\begin{proof}
Let $T$ be a subdivision of $\Gamma_k$ in $G$.  We use the well-known fact that for all $e \in E(G)$, at least one of $G \backslash e$ or $G / e$ is $2$-connected.  Therefore, if some edge $e$ of $G$ has an end not in $V(T)$, then $G \backslash e$ or $G / e$ is a $2$-connected graph containing a subdivision of $\Gamma_k$, which contradicts the minor-minimality of $G$.
\end{proof}

\begin{lemma}
\label{lem:marked}
Let $1\leq \ell \leq k$ and let $T$ be a tree isomorphic to $\Gamma_k$ with root $r$.  
Suppose that a non-empty subset of vertices of $T$ are marked. 
Then
\begin{enumerate}
    \item\label{item:T-with-all-leaves-marked} $T$ contains a subdivision of $\Gamma_\ell$, all of whose leaves are marked, or 
    \item\label{item:T-with-a-vertex-with-marked-vertices-only-one-way} there exist a vertex $v\in V(T)$ and a child $w$ of $v$ such that $T_v$ has at least one marked vertex but $T_w$ has none, and $w$ is at distance at most $\ell$ from $r$. 
\end{enumerate}
\end{lemma}

\begin{proof}
A vertex $v$ in $T$ is \emph{good} if there is a marked vertex in $T_v$, and is {\em bad} otherwise. 
Let $T' $ be the subtree of $T$ induced by vertices at distance at most $\ell$ from $r$ in $T$. 
If each leaf of $T'$ is good, then for each such leaf $u$ we can find a marked vertex $m_u$ in $T_u$, and $T' \cup\bigcup\set{uTm_u : \text{$u$ leaf of $T' $}}$ is a $\Gamma_\ell$ subdivision with all leaves marked, as required by \ref{item:T-with-all-leaves-marked}.
Now assume that some leaf $u$ of $T'$ is bad. Let $w$ be the bad vertex closest to $r$ on the $rTu$ path. Since some vertex in $T$ is marked, $r$ is good. Thus $w\neq r$. 
Moreover, the parent $v$ of $w$ is good, by our choice of $w$. 
Also, $w$ is at distance at most $\ell$ from $r$. 
Therefore, $v$ and $w$ satisfy~\ref{item:T-with-a-vertex-with-marked-vertices-only-one-way}.
\end{proof}

Our main technical tools are~\ref{lem:new_special_path} and~\ref{lem:attaching_to_special_path} below, which are lemmas about $2$-connected graphs $G$ containing a subdivision $T$ of $\Gamma_k$ as a spanning tree.  
In order to state them, we need to introduce some definitions and notation.  

For the next two paragraphs, let $G$ be a $2$-connected graph containing a subdivision $T$ of $\Gamma_k$ as a spanning tree.  
For each vertex $v\in V(G)$, let $\height(v)$ be the number of original non-leaf vertices on the path $vTw$, where $w$ is any leaf of $T_v$.  We stress the fact that \emph{subdivision vertices are not counted} when computing $\height(v)$.   Since the length of a path in $\Gamma_k$ from a fixed vertex to any leaf is the same, $\height(v)$ is independent of the choice of $w$.   
We also use the shorthand notation $\OUT(v):= V(G) \setminus V(T_v)$ when $G$ and $T$ are clear from the context. For $X,Y \subseteq V(G)$, we say that $X$ \emph{sees} $Y$ if $xy \in E(G)$ for some $x \in X$ and $y \in Y$.
If $P$ is a path with ends $x$ and $y$, and $Q$ is a path with ends $y$ and $z$, then let $PQ$ be the walk that follows $P$ from $x$ to $y$ and then follows $Q$ from $y$ to $z$.

A path $P$ of $G$ is {\em $(x,a,y)$-special} if $|V(P)| \geq 3$, and $x,y$ are the ends of $P$, and $a$ is a child of $x$ such that $V(P) \setminus \{x,y\} \subseteq V(T_a)$ and $y\notin V(T_a)$.  
A vertex $w$ is {\em safe} for an $(x,a,y)$-special path $P$ if $w$ satisfies the following properties:
\begin{itemize}
    \item the parent $v$ of $w$ is in $V(P) \setminus \{x,y\}$;
    \item $\height(v) \geq \height(x) - 2\ell$; 
    \item $V(P) \cap V(T_{w}) = \emptyset$;  
    \item $V(T_{w})$ does not see $\OUT(a) \setminus \{x\}$, and
    \item if $v$ is an original vertex and $u$ is its child distinct from $w$, then either $V(P) \cap V(T_{u}) \neq \emptyset$ or $V(T_{u})$ does not see $\OUT(a) \setminus \{x\}$.
\end{itemize}

\begin{lemma}
\label{lem:new_special_path} 
Let $1 \leq \ell \leq k$. 
Let $G$ be a minor-minimal $2$-connected graph containing a subdivision of $\Gamma_k$.  Let $T$ be a subdivision of $\Gamma_k$ in $G$,
$v \in V(T)$ with $\height(v) \geq 3\ell+1$, and $w$ be a child of $v$. 
Then, either $G$ contains a $\Gamma_{\ell}^{+}$ minor, or there is a $(v_0, w_0, v'_0)$-special path $P$ and two distinct safe vertices for $P$ such that:
\begin{itemize}
    \item $V(P) \subseteq V(T_{w})$, 
    \item $\height(v_0) \geq \height(v) - \ell$, 
    \item $V(T_{v_0})$ sees $\OUT(w) \setminus \{v\}$, 
    \item $V(T_{w_0})$ does not see $\OUT(w) \setminus \{v\}$, and
    \item $V(T_{u_0})$ sees $\OUT(v_0)$ if $v_0$ is an original vertex and $u_0$ is its child distinct from $w_0$. 
\end{itemize}
\end{lemma}
\begin{proof}
By~\ref{lem:spanning}, $T$ is a spanning tree of $G$.  Colour red each vertex of $T_{w}$ that sees a vertex in $\OUT(w) \setminus \{v\}$. 
Observe that there is at least one red vertex. 
Indeed, $V(T_w)$ must see $\OUT(w) \setminus \{v\}$, for otherwise $v$ would be a cut vertex separating $V(T_{w})$ from $\OUT(w) \setminus \{v\}$ in $G$. 

Let $\tilde T_w$ be the complete binary tree obtained from $T_w$ by iteratively contracting each edge of the form $pq$ with $p$ a subdivision vertex and $q$ the child of $p$ into vertex $q$. 
Declare $q$ to be coloured red after the edge contraction if at least one of $p,q$ was coloured red beforehand. 
Since $\height(w) \geq \height(v)-1 \geq 3\ell$, the tree $\tilde T_w$ has height at least $3\ell$. 

If $\tilde T_{w}$ contains a subdivision of $\Gamma_\ell$ with all leaves coloured red, then so does $T_w$. Therefore, $G$ contains $\Gamma_{\ell}^{+}$ as a minor, because $\OUT(w)$ induces a connected subgraph of $G$ which is vertex-disjoint from $V(T_w)$ and which sees all the leaves of $T_w$.  
Thus, by~\ref{lem:marked}, we may assume there is a vertex $\tilde v_0$ of $\tilde T_{w}$ and a child $\tilde w_0$ of $\tilde v_0$ with $\height(\tilde w_0) \geq \height(w) - \ell$ such that $T_{\tilde v_0}$ has at least one red vertex but $T_{\tilde w_0}$ has none. 
Going back to $T_w$, we deduce that there is a vertex $v_0$ of $T_{w}$ and a child $w_0$ of $v_0$ with $\height(w_0) \geq \height(w) - \ell$ such that $T_{v_0}$ has at least one red vertex but $T_{w_0}$ has none. 
To see this, choose $v_0$ as the deepest red vertex in the preimage of $\tilde v_0$.
Note that $v_0$ or $w_0$ could be subdivision vertices. 

If $v_0$ is an original vertex, let $u_0$ denote the child of $v_0$ distinct from $w_0$. 
Since $v_0$ is not a cut vertex of $G$, one of the two subtrees $T_{u_0}$ and $T_{w_0}$ sees $\OUT(v_0)$.
If $T_{u_0}$ does not see $\OUT(v_0)$, then $T_{u_0}$ has no red vertex and $T_{w_0}$ sees $\OUT(v_0)$.  Therefore, by exchanging $u_0$ and $w_0$ if necessary, we guarantee that the following two properties hold when $u_0$ exists. 
\begin{equation}
\textrm{$T_{u_0}$ sees $\OUT(v_0)$ \qquad\quad and \qquad\quad $T_{w_0}$ has no red vertex.}  \label{eq:u0} 
\end{equation}

We iterate this process in $T_{w_0}$. 
Colour blue each vertex of $T_{w_0}$ that sees a vertex in $\OUT(w_0) \setminus \{v_0\}$. 
There is at least one blue vertex, since otherwise $v_0$ would be a cut vertex of $G$ separating $V(T_{w_0})$ from $\OUT(w_0) \setminus \{v_0\}$.  
Defining $\tilde T_{w_0}$ similarly as above, if $\tilde T_{w_0}$ contains a subdivision of $\Gamma_\ell$ with all leaves coloured blue, then $G$ has a $\Gamma_{\ell}^{+}$ minor. 
Applying~\ref{lem:marked} and going back to $T_{w_0}$, we may assume there is a vertex $v_1$ of $T_{w_0}$ and a child $w_1$ of $v_1$ with $\height(w_1) \geq \height(w_0) - \ell$ such that $T_{v_1}$ has at least one blue vertex but $T_{w_1}$ has none. 

We now define the $(v_0, w_0, v'_0)$-special path $P$, and identify two distinct safe vertices for $P$. 
To do so, we will need to consider different cases. 
In all cases, the end $v'_0$ will be a vertex of $\OUT(w_0) \setminus \{v_0\}$ seen by a (carefully chosen) blue vertex in $T_{v_1}$, thus $v'_0 \notin V(T_{w_0})$, and the path $P$ will be such that $V(P) \setminus \{v_0,v'_0\} \subseteq V(T_{w_0})$.  
Note that the end $v_0$ of $P$ satisfies 
$\height(v_0)  \geq \height(v) - \ell$, as desired.  

Before proceeding with the case analysis, we point out the following property of $G$.
If $st$ is an edge of $G$ such that $G/st$ contains a subdivision of $\Gamma_k$, then $G/st$ is not $2$-connected by the minor-minimality of $G$, and it follows that $\{s,t\}$ is a cutset of $G$. 
Note that this applies if $st$ is an edge of $T$ such that at least one of $s, t$ is a subdivision vertex, or if $st$ is an edge of $E(G) \setminus E(T)$ linking two subdivision vertices of $T$ that are on the same subdivided path of $T$.
This will be used below.

\textbf{Case~1.} $v_1$ is a subdivision vertex: \\
In this case, $v_1$ is the unique blue vertex in $T_{v_1}$. 
Let $v'_0$ be a vertex of $\OUT(w_0)\setminus\{v_0\}$ seen by the blue vertex $v_1$. 
Since $v_1$ is not a cut vertex of $G$, there is an edge $st$ with $s\in V(T_{w_1})$ and $t\in \OUT(v_1)$. 
Note that $t\in V(T_{w_0}) \cup \{v_0\}$, since $T_{w_1}$ has no blue vertex. 

\textbf{Case~1.1.} There is at least one original vertex on the path $v_1Ts$: \\ 
Let $q$ be the first original vertex on the path $v_1Ts$. 
Let $s_1$ denote a child of $q$ not on the $qTs$ path. 
Let $q'$ be the first original vertex distinct from $q$ on the $qTs$ path if any, and otherwise let $q':=s$ (note that possibly $q'=q=s$).  
Let $s_2$ be a child of $q'$ not on the $qTs$ path, and distinct from $s_1$ if $q'=q$.  
As illustrated in Figure~\ref{fig:hmmm}, define
\[
P:=v_0TtsTv_1v'_0.
\]
Observe that $V(P)\setminus\{v_0,v'_0\} \subseteq V(T_{w_0})$, by construction. 
Observe also that the parent $q'$ of $s_2$ satisfies $\height(q') \geq \height(q)-1 = \height(v_1)-1 \geq \height(v_0) - \ell - 1 \geq \height(v_0) - 2\ell$. 
It can be checked that $s_1, s_2$ are two distinct safe vertices for $P$, as desired. 

\textbf{Case~1.2.} All vertices of the path $v_1Ts$ are subdivision vertices: \\
In particular, $w_1$ is a subdivision vertex. 
We show that the unique child $q$ of $w_1$ is an original vertex, and therefore $s=w_1$.  
Indeed, assume not, and let $q'$ denote the child of $q$.   
Since $v_1$ is not a cut vertex of $G$ but $\{v_1,w_1\}$ is a cutset of $G$, we deduce that $w_1$ sees a vertex $w'_1$ in $\OUT(v_1)$ and that  $V(T_{q})$ does not see $\OUT(v_1)$. 
Similarly, because $w_1$ is not a cut vertex of $G$ but $\{w_1,q\}$ is a cutset of $G$, we deduce that $qv_1 \in E(G)$ and that  $V(T_{q'})$ does not see $\OUT(w_1)$. 
Since $q$ is not a cut vertex, some vertex $q'' \in V(T_{q'})$ sees $\OUT(q)$, and hence sees $w_1$ (since $V(T_{q'})$ does not see $\OUT(v_1)$). 
But then, because of the edges $q''w_1$ and $w_1w'_1$, we see that $\{v_1,q\}$ cannot be a cutset of $G$. 
It follows that $G/v_1q$ is $2$-connected and contains a $\Gamma_k$ minor, contradicting our assumption on $G$. 

Hence, $q$ is an original vertex, and $s=w_1$. 
Since $w_1$ is not a cut vertex of $G$, there is an edge linking $V(T_q)$ to $\OUT(w_1)$. 
Since $\{v_1, w_1\}$ is a cutset of $G$, this edge links some vertex $s' \in V(T_q)$ to $v_1$. 

Let $s_1$ denote a child of $q$ not on the $qTs'$ path. 
Let $q'$ be the first original vertex distinct from $q$ on the $qTs'$ path if any, and otherwise let $q':=s'$ (note that possibly $q'=s'=q$).  
Let $s_2$ be a child of $q'$ not on the $qTs'$ path, and distinct from $s_1$ if $q'=q$.  
As illustrated in Figure~\ref{fig:hmmm}, define 
\[
P:=v_0Ttw_1Ts'v_1v'_0. 
\]
Again, note that $V(P)\setminus\{v_0,v'_0\} \subseteq V(T_{w_0})$ by construction. 
Observe also that the parent $q'$ of $s_2$ satisfies $\height(q') \geq \height(q)-1 = \height(v_1)-1 \geq \height(v_0) - \ell - 1 \geq \height(v_0) - 2\ell$. 
It is easy to see that $s_1, s_2$ are two distinct safe vertices for $P$, as desired. 
\begin{figure}[!ht]
\centering
\includegraphics{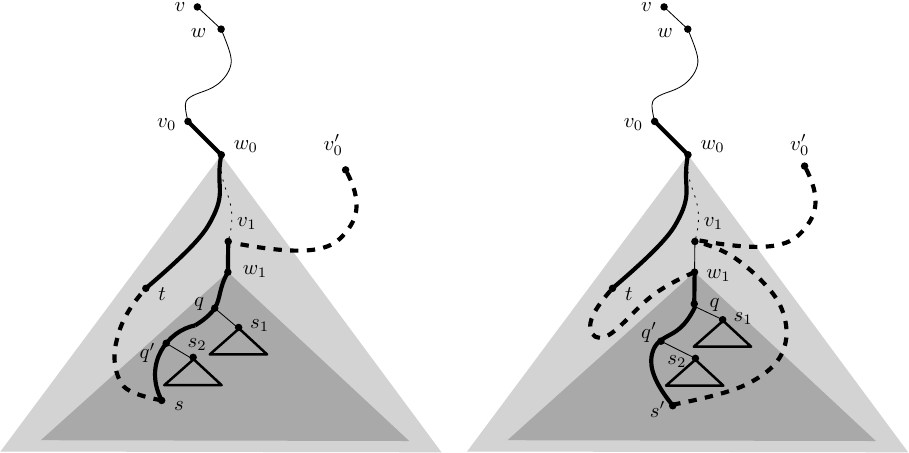}
\caption{Path $P$ and the safe vertices $s_1,s_2$. Cases 1.1 and 1.2}
\label{fig:hmmm}
\end{figure}

\textbf{Case~2.} $v_1$ is an original vertex: \\ 
Let $u_1$ denote the child of $v_1$ distinct from $w_1$. 
If $T_{u_1}$ has no blue vertex, then $v_1$ is the unique blue vertex in $T_{v_1}$. 
Let $v'_0$ be a vertex of $\OUT(w_0)\setminus\{v_0\}$ seen by the blue vertex $v_1$. 
Define
\[
P:=v_0Tv_1v'_0.  
\]
Clearly,  $V(P)\setminus\{v_0,v'_0\} \subseteq V(T_{w_0})$,  and $u_1, w_1$ are two distinct safe vertices for $P$.  

Next, assume that $T_{u_1}$ has a blue vertex. 
In this case, we need to define an extra pair $v_2,w_2$ of vertices. 
Observe that $\height(u_1) \geq \height(w_0) - \ell \geq \height(w) - 2\ell = \height(v) - 2\ell -1 \geq \ell$. 
Let $\tilde T_{u_1}$ be the tree obtained from $T_{u_1}$, as before.  
Again, if $\tilde T_{u_1}$ contains a subdivision of $\Gamma_\ell$ all of whose leaves are blue, then $G$ contains an $\Gamma_\ell^+$ minor.  
Thus, by~\ref{lem:marked}, we may assume there is a vertex $v_2$ of $T_{u_1}$ and a child $w_2$ of $v_2$ with $\height(w_2) \geq \height(u_1) - \ell = \height(w_1) - \ell$ such that $T_{v_2}$ has at least one blue vertex, but $T_{w_2}$ has none.  

\textbf{Case~2.1.} $v_2$ is a subdivision vertex: \\  
Here, $v_2$ is the unique blue vertex in $T_{v_2}$. 
Let $v'_0$ be a vertex of $\OUT(w_0)\setminus\{v_0\}$ seen by $v_2$. 
As illustrated in Figure~\ref{fig:hmmm-2}, define
\[
P:=v_0Tv_2v'_0. 
\]
Observe that  $V(P)\setminus\{v_0,v'_0\} \subseteq V(T_{w_0})$ by construction, and that $w_1, w_2$ are two distinct safe vertices for $P$.

\textbf{Case~2.2.} $v_2$ is an original vertex: \\ 
Let $u_2$ be the child of $v_2$ distinct from $w_2$. 
Let $b_2$ denote a blue vertex in $V(T_{u_2}) \cup \{v_2\}$, distinct from $v_2$ if possible.   
Let $v'_0$ be a vertex of $\OUT(w_0)\setminus\{v_0\}$ seen by the blue vertex $b_2$. 
Define 
\[
P:=v_0Tb_2v'_0.  
\]
Again,  $V(P)\setminus\{v_0,v'_0\} \subseteq V(T_{w_0})$ by construction. 

If $b_2 \neq v_2$, then $P$ intersects $V(T_{u_2})$. 
If $b_2 = v_2$, then $P$ avoids $V(T_{u_2})$, and $V(T_{u_2})$ has no blue vertex. That is, $V(T_{u_2})$ does not see $\OUT(w_0) \setminus \{v_0\}$.  
Using these observations, one can check that $w_1, w_2$ are two distinct safe vertices for $P$ in both cases; see Figure~\ref{fig:hmmm-2}. \qedhere

\begin{figure}[!ht]
\centering
\includegraphics{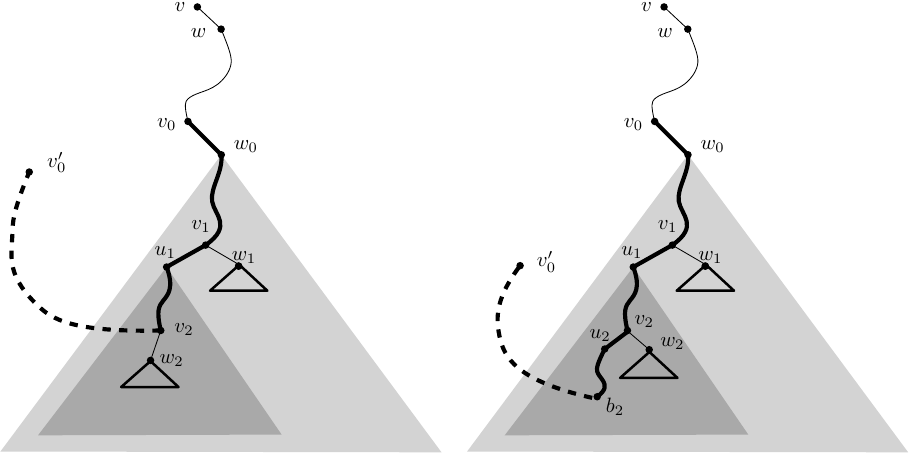}
\caption{Path $P$ and the safe vertices $w_1,w_2$. Cases 2.1 and 2.2}
\label{fig:hmmm-2}
\end{figure}
\end{proof}

\begin{lemma}
\label{lem:attaching_to_special_path} 
Let $1 \leq \ell \leq k$. 
Let $G$ be a minor-minimal $2$-connected graph containing a subdivision of $\Gamma_k$ and let $T$ be a subdivision of $\Gamma_k$ in $G$.
Let $S$ be an $(x,a,y)$-special path with $\height(x) \geq 5\ell+1$.  
Let $w$ be a safe vertex for $S$ and let $v\in V(S)$ denote the parent of $w$ in $T$.    
Then, either $G$ contains a $\Gamma_{\ell}^{+}$ minor, or there is a $(v_0, w_0, v'_0)$-special path $P$, two distinct safe vertices $w_1,w_2$ for $P$, and an $S$-ear $Q$ such that: 
\begin{enumeratea}
    \item $V(P) \subseteq V(T_{w})$, \label{prop:in_safe_subtree}
    \item $\height(v_0) \geq \height(x) - 3\ell$, \label{prop:height}
    \item $V(T_{w_0})$ does not see $\OUT(w) \setminus \{v\}$, \label{prop:does_not_see}
    \item $P \subseteq Q$, \label{prop:PsubsetQ}
    \item $V(Q) \setminus V(P) \subseteq \OUT(w_0) \setminus \{v_0\}$,  \label{prop:QP}
    \item $V(Q) \subseteq V(T_a) \cup \{x\}$,  \label{prop:Qax} 
    \item $V(Q) \cap V(T_{w_i})=\emptyset$ for $i=1,2$, and \label{prop:Qsi}
    \item if $e \in E(Q) \setminus E(T)$ and no end of $e$ is in $V(T_w)$, then $v$ is an original vertex with children $u,w$, the path $S$ is disjoint from $V(T_u)$, and $e$ links $V(T_u)$ to $\OUT(v)$. \label{prop:edge_e}
\end{enumeratea}
\end{lemma}

\begin{proof}
By~\ref{lem:spanning}, $T$ is a spanning tree.  Also, $G$ does not contain $\Gamma_{\ell}^{+}$ as a minor (otherwise, we are done). 
Applying~\ref{lem:new_special_path} on vertex $v$ and its child $w$, we obtain a $(v_0, w_0, v'_0)$-special path $P$ and two distinct safe vertices $w_1,w_2$ for $P$ such that $V(P) \subseteq V(T_{w})$; 
$\height(v_0) \geq \height(v) - \ell \geq \height(x) - 3\ell$;  
$V(T_{v_0})$ sees $\OUT(w) \setminus \{v\}$; 
$V(T_{w_0})$ does not see $\OUT(w) \setminus \{v\}$; 
and if $v_0$ is an original vertex and $u_0$ is the child of $v_0$ distinct from $w_0$ then $V(T_{u_0})$ sees $\OUT(v_0)$. 
It remains to extend $P$ into an $S$-ear $Q$ satisfying properties~\ref{prop:PsubsetQ}--\ref{prop:edge_e}. 
The proof is split into twelve cases, all of which are illustrated in Figure~\ref{fig:hmmm-3}.
\begin{figure}[!ht]
\centering
\includegraphics{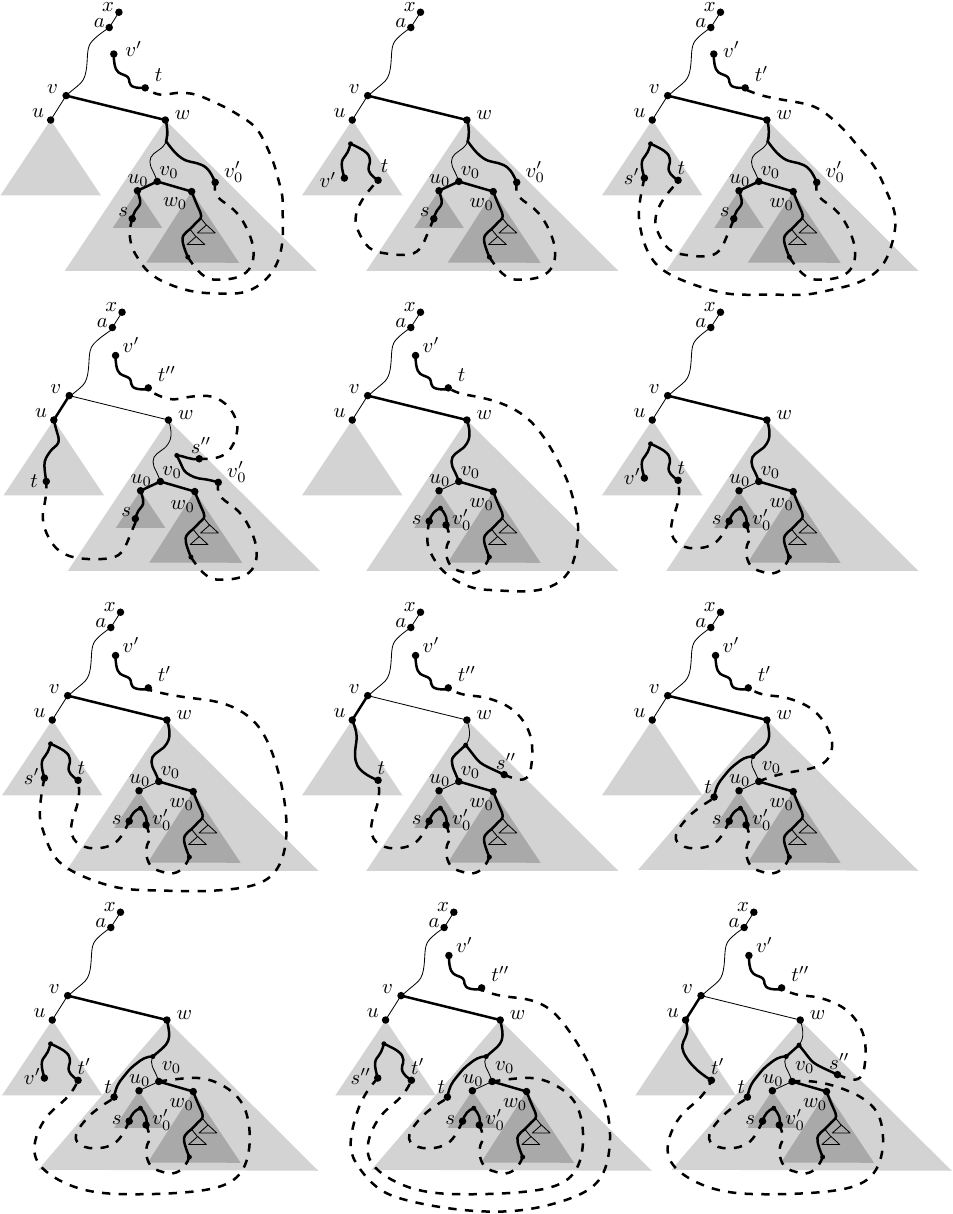}
\caption{Definition of $S$-ears $Q_1, \dots, Q_{12}$.}
\label{fig:hmmm-3}
\end{figure}

If $v$ is an original vertex, let $u$ denote the child of $v$ distinct from $w$. 
In order to simplify the arguments below, we let $V(T_{u})$ be the empty set if $u$ is not defined (same for $u_0$). 

First assume that $v'_0 \not\in V(T_{u_0})$. 
Then $v'_0 \in \OUT(v_0) \cap V(T_{w})$.  
Recall that $V(T_{v_0}) \setminus V(T_{w_0}) = V(T_{u_0})\cup\set{v_0}$ sees $\OUT(w) \setminus \set{v} = V(T_u)\cup\OUT(v)$.
Suppose that there is an edge $st\in E(G)$ with $s\in V(T_{u_0})\cup\set{v_0}$ and $t\in\OUT(v)$. 
Note that $t\in V(T_a)\cup\set{x}$, since $w$ is a safe vertex for $S$.
Let $v'$ be the closest ancestor of $t$ in $T$ that lies on $S$.
Note that $v' \in V(T_a)\cup \set{x}$.
Define 
\[
Q_1:= vTv_0'Pv_0TstTv'.
\]
Next, suppose that there is no such edge $st$. 
Then, there must be an edge $st$ with $s\in V(T_{u_0})\cup\set{v_0}$ and $t\in V(T_{u})$.
In particular, $u$ exists.
If the path $S$ intersects $V(T_u)$, then let $v'$ be a vertex in $V(S)\cap V(T_u)$ that is closest to $t$ in $T$. Define
\[
Q_2:=vTv_0'Pv_0TstTv'.
\]
Otherwise, we have $V(S)\cap V(T_u)=\emptyset$.
Since $w$ is a safe vertex for $S$,  $V(T_u)$ does not see $\OUT(a) \setminus \set{x}$ in this case.
If $V(T_u)$ sees $\OUT(v)$, then let $s't'$ be an edge with $s'\in V(T_u)$ and $t'\in \OUT(v)$, and let $v'$ be the closest ancestor of $t'$ in $T$ that lies on $S$.
Note that both $t'$ and $v'$ lie in $V(T_a)\cup \set{x}$.
Define
\[
Q_3:=vTv_0'Pv_0TstTs't'Tv'.
\]
Otherwise, $V(T_u)$ does not see $\OUT(v)$.
Since $v$ is not a cut vertex in $G$, we deduce that $V(T_w)$ sees $\OUT(v)$.
As we already know that neither $V(T_{w_0})$ nor $V(T_{u_0})\cup\set{v_0}$ sees $\OUT(v)$, there is an edge $s''t'' \in E(G)$ with $s''\in V(T_w) \setminus V(T_{v_{0}})$ and $t''\in \OUT(v)$.
Again, since $w$ is safe for $S$, we know that $t''\in V(T_a)\cup\set{x}$.
Let $v'$ be the closest ancestor of $t''$ in $T$ that lies on $S$.
Note that $v' \in V(T_a)\cup \set{x}$.
Define
\[
Q_4:=vTtsTv_0Pv_0'Ts''t''Tv'.
\]

Next, assume that $v_0'\in V(T_{u_0})$. 
In particular, $u_0$ exists.
Recall that $V(T_{u_0})$ sees $\OUT(v_0)$.
If $V(T_{u_0})$ sees $\OUT(v)$ then 
let $st$ be an edge with $s\in V(T_{u_0})$ and $t\in \OUT(v)$.
Observe that $t\in V(T_a)\cup\set{x}$ since $w$ is safe for $S$. 
Let $v'$ be the closest ancestor of $t$ in $T$ that lies on $S$.
Note that $v'\in V(T_a)\cup \set{x}$ as well.
Define
\[
Q_5:=vTv_0Pv'_0TstTv'.
\]
Next, suppose that $V(T_{u_0})$ does not see $\OUT(v)$. 
If $V(T_{u_0})$ sees $V(T_u)$, then
let $st$ be an edge with $s\in V(T_{u_0})$ and $t\in V(T_u)$.
In particular, $u$ exists.
If $S$ intersects $V(T_u)$, then let $v'$ be a vertex in $V(S)\cap V(T_u)$ that is closest to $t$ in $T$.
Define
\[
Q_6:=vTv_0Pv'_0TstTv'.
\]
Otherwise, we have $V(S)\cap V(T_u)=\emptyset$.
Since $w$ is a safe vertex for $S$, $V(T_u)$ does not see $\OUT(a) \setminus \set{x}$ in this case. 
If $V(T_u)$ sees $\OUT(v)$, then
let $s't'$ be an edge with $s'\in V(T_u)$ and $t'\in \OUT(v)$ and 
let $v'$ be the closest ancestor of $t'$ in $T$ that lies on $S$.
Note that both $t'$ and $v'$ lie in $V(T_a)\cup \set{x}$.
Define
\[
Q_7:=vTv_0Pv'_0TstTs't'Tv'.
\]
Next, suppose that $V(T_u)$ does not see $\OUT(v)$.
Since $v$ is not a cut vertex in $G$, we deduce that $V(T_w)$ sees $\OUT(v)$.
As we already know that neither $V(T_{w_0})$ nor $V(T_{u_0})$ sees $\OUT(v)$, 
there is an edge $s''t'' \in E(G)$ with $s''\in (V(T_w) \setminus V(T_{v_{0}}))\cup\set{v_0}$ and $t''\in \OUT(v)$.
Again, since $w$ is safe for $S$, $t''\in V(T_a)\cup\set{x}$.
Let $v'$ be the closest ancestor of $t''$ in $T$ that lies on $S$.
Note that $v' \in V(T_a)\cup \set{x}$.
Define
\[
Q_8:=vTtsTv'_0Pv_0Ts''t''Tv'.
\]
We are done with the cases where $V(T_{u_0})$ sees $\OUT(v)$ or $V(T_u)$. 
Next, assume that $V(T_{u_0})$ sees neither of these two sets. 
Since $V(T_{u_0})$ sees $\OUT(v_0)$, there is an edge $st$ with $s\in V(T_{u_0})$ and $t\in V(T_w) \setminus V(T_{v_0})$.
Recall that $V(T_{v_0})$ sees $\OUT(w) \setminus \set{v}$.
Since neither $V(T_{u_0})$ nor $V(T_{w_0})$ sees $\OUT(w) \setminus \set{v}$, we conclude that $v_0$ sees $\OUT(w) \setminus \set{v}$.
If $v_0$ sees $\OUT(v)$, then let $v_0t'$ be an edge with $t'\in \OUT(v)$.
Let $v'$ be the closest ancestor of $t'$ in $T$.
As before, $\{t',v'\} \subseteq V(T_a)\cup\set{x}$.
Define
\[
Q_9:=vTtsTv'_0Pv_0t'Tv'.
\]
Otherwise, $v_0$ sees $V(T_u)$.
Let $v_0t'$ be an edge with $t'\in V(T_u)$.
If  $S$ intersects $V(T_u)$, then let $v'$ be a vertex in $V(S)\cap V(T_u)$ that is closest to $t'$ in $T$. Define
\[
Q_{10}:=vTtsTv'_0Pv_0t'Tv'.
\]
Otherwise, $V(S)\cap V(T_u)=\emptyset$.
Since $w$ is a safe vertex for $S$, we know that $V(T_u)$ does not see $\OUT(a) \setminus \set{x}$ in this case.
If $V(T_u)$ sees $\OUT(v)$, then
let $s''t''$ be an edge with $s''\in V(T_u)$ and $t''\in \OUT(v)$ and 
let $v'$ be the closest ancestor of $t''$ in $T$ that lies on $S$.
Note that both $t''$ and $v'$ lie in $V(T_a)\cup \set{x}$.
Define
\[
Q_{11}:=vTtsTv'_0Pv_0t'Ts''t''Tv'.
\]
Otherwise, $V(T_u)$ does not see $\OUT(v)$.
Since $v$ is not a cut vertex in $G$, we deduce that $V(T_w)$ sees $\OUT(v)$.
As we already know that neither $V(T_{w_0})$ nor $V(T_{u_0})\cup \set{v_0}$ sees $\OUT(v)$, 
there is an edge $s''t'' \in E(G)$ with $s''\in V(T_w) \setminus V(T_{v_{0}})$ and $t''\in \OUT(v)$.
Again, since $w$ is safe for $S$, $t''\in V(T_a)\cup\set{x}$.
Let $v'$ be the closest ancestor of $t''$ in $T$ that lies on $S$.
Note that $v' \in V(T_a)\cup \set{x}$.
Define
\[
Q_{12}:=vTt'v_0Pv'_0TstTs''t''Tv'.
\]

One can check that for all $i \in [12]$, if we set $Q=Q_i$, then $Q$ is an $S$-ear satisfying  properties~\ref{prop:PsubsetQ}--\ref{prop:edge_e}.  
\end{proof}

We now prove~\ref{lem:findingbeds} using~\ref{lem:new_special_path} and~\ref{lem:attaching_to_special_path}. 

\begin{proof}[Proof of~\ref{lem:findingbeds}] 
Let $T$ be a subdivision of $\Gamma_k$ in $G$, which is a spanning tree of $G$ by~\ref{lem:spanning}. 
Also, $G$ has no $\Gamma_{\ell}^{+}$ minor (otherwise, we are done). 
As before, for $v\in V(G)$, we let $\height(v)$ be the number of original non-leaf vertices on the path $vTw$, where $w$ is any leaf of $T_v$.
The \emph{depth} of $x \in V(T^1)$, denoted $\depth(x)$, is the number of edges in $xT^1r$, where $r$ is the root of $T^1$.

We prove the stronger statement that $G$ contains a clean binary pear tree $(T^1, \{(P_x, Q_x):x \in V(T^1)\})$ such that: 
\begin{enumeratearabic}
    \item for all  $x \in V(T^1)$, the path $P_x$ is a $(v_x, w_x, v'_x)$-special path for some vertices $v_x, w_x, v_x'$ of $G$ such that $\height(v_x) \geq k - 3\ell \depth(x) - \ell$, and $P_x$ has two distinguished safe vertices; moreover,  if $x$ is not a leaf we associate these safe vertices with the two children $y,z$ of $x$ and denote them $s_{xy}$ and $s_{xz}$; \label{prop:height_vx}     
    \item  for all  $x,y\in V(T^1)$,  
    $v_x$ is an ancestor of $v_y$ in $T$ if and only if $x$ is an ancestor of $y$ in $T^1$; \label{prop:kind-of-obvious} 
    \item for all  $x,y\in V(T^1)$ such that $y$ is a child of $x$, 
    the paths $P_y$ and $Q_y$ are obtained by applying~\ref{lem:attaching_to_special_path} on $P_x$ with safe vertex $s_{xy}$;  \label{prop:obtained_from_lemma}
    \item for all  $y,z\in V(T^1)$ such that $y$ and $z$ are siblings, 
    no vertex of $Q_z$ meets $T_{w_y}$, and
    no vertex of $Q_y$ meets $T_{w_z}$;  \label{prop:sibling}
    \item for all leaves $x$ of $T^1$,
    $V(T_{w_x})$ and $\bigcup_{p\in V(T^1)\setminus \set{x}} V(Q_p)$ are disjoint.\label{prop:new}
\end{enumeratearabic}

The proof is by induction on $|V(T^1)|$. 
For the base case $|V(T^1)|=1$, the tree $T^1$ is a single vertex $x$.   
Applying~\ref{lem:new_special_path} with $v$ the root of $T$ and $w$ a child of $v$ in $T$, we obtain a $(v_x, w_x, v'_x)$-special path $P_x$ and two distinct safe vertices for $P_x$. 
Let $Q_x:=P_x$. 
Then $(T^1, \{(P_x,Q_x)\})$ is a binary pear tree in $G$. 
Observe that $\depth(x)=0$ and $\height(v_x) \geq \height(v) - \ell = k - \ell$, thus~\ref{prop:height_vx} holds. 
Properties~\ref{prop:kind-of-obvious}--\ref{prop:new} hold vacuously since $x$ is the only vertex of $T^1$. 

Next, for the inductive case, assume $|V(T^1)|>1$. 
Let $x$ be a vertex of $T^1$ with two children $y,z$ that are leaves of $T^1$. 
Applying induction on the binary tree $T^1 - \{y,z\}$,
we obtain a binary pear tree $(T^1-\{y,z\}, \{(P_p, Q_p):p \in V(T^1 - \{y,z\})\})$ in $G$ satisfying the claim. 

Note that $\depth(x)\leq 3\ell-3$, and thus 
$\height(v_x) \geq k - 3\ell \depth(x) - \ell \geq (9\ell^2 - 3\ell +1) - 3\ell (3\ell-3) - \ell \geq 5\ell +1$. 
By~\ref{prop:height_vx}, the path $P_x$ comes with two distinguished safe vertices. 
Considering now the two children $y, z$ of $x$ in the tree $T$, we associate these safe vertices to $y$ and $z$, as expected, and denote them $s_{xy}$ and $s_{xz}$.  
Let $v_{xy}$ and $v_{xz}$ denote their respective parents in $T$. 
First, apply~\ref{lem:attaching_to_special_path} with the path $P_x$ and safe vertex $s_{xy}$, giving a $(v_{y}, w_{y}, v'_{y})$-special path $P_{y}$ with two distinct safe vertices, and a $P_x$-ear $Q_{y}$ satisfying the properties of~\ref{lem:attaching_to_special_path}. 
Next, apply~\ref{lem:attaching_to_special_path} with the path $P_x$ and safe vertex $s_{xz}$, giving a $(v_{z}, w_{z}, v'_{z})$-special path $P_{z}$ with two distinct safe vertices, and a $P_x$-ear $Q_{z}$ satisfying the properties of~\ref{lem:attaching_to_special_path}. 

Observe that, by property~\ref{prop:height} of~\ref{lem:attaching_to_special_path}, $\height(v_{y}) \geq \height(v_x) - 3\ell \geq k - 3\ell \depth(x) - 4\ell = k - 3\ell \depth(y) - \ell$, and similarly  
$\height(v_{z}) \geq k - 3\ell \depth(z) - \ell$. 
Thus, property~\ref{prop:height_vx} is satisfied. 
Clearly, property~\ref{prop:kind-of-obvious} and property~\ref{prop:obtained_from_lemma} are satisfied as well. 
To establish property~\ref{prop:sibling}, it only remains to show that no vertex of $Q_z$ meets $T_{w_y}$, and that no vertex of $Q_y$ meets $T_{w_z}$.
By symmetry it is enough to show the former, which we do now. 

Arguing by contradiction, assume that $Q_{z}$ meets $T_{w_{y}}$. 
Since $V(T_{w_{y}}) \subseteq V(T_{s_{xy}})$ and $V(Q_x) \cap V(T_{s_{xy}}) = \emptyset$ (by property~\ref{prop:Qsi} of~\ref{lem:attaching_to_special_path}), and since the two ends of $Q_{z}$ are on $Q_x$, we see that the two ends of $Q_{z}$ are outside $V(T_{w_{y}})$. 
Thus, at least two edges of $Q_{z}$ have exactly one end in $V(T_{w_{y}})$, and there is an edge $st$ which is not an edge of $T$ (i.e.\ $st \neq v_yw_y$). 
By symmetry, $s\in V(T_{w_{y}})$ and $t\notin V(T_{w_{y}})$. 

Clearly, $s\notin V(T_{s_{xz}})$ since $V(T_{w_{y}}) \subseteq V(T_{s_{xy}})$, and $V(T_{s_{xy}}) \cap V(T_{s_{xz}})=\emptyset$.  
Moreover, $t\notin V(T_{s_{xz}})$, since $V(T_{s_{xz}}) \subseteq \OUT(s_{xy}) \setminus \{v_{xy}\}$ and since $V(T_{w_{y}})$ does not see $\OUT(s_{xy}) \setminus \{v_{xy}\}$ by property~\ref{prop:does_not_see} of~\ref{lem:attaching_to_special_path}. 
Since $st$ is an edge of $Q_z$ not in $T$ with neither of its ends in $V(T_{s_{xz}})$, it follows from property~\ref{prop:edge_e} of~\ref{lem:attaching_to_special_path}  that $v_{xz}$ is an original vertex with children $u_{xz}$ and $s_{xz}$; the path $P_x$ avoids $V(T_{u_{xz}})$; and the edge $st$ has one end in $V(T_{u_{xz}})$ and the other in $\OUT(v_{xz})$. 
(We remark that we do not know which end is in which set at this point.) 

First, suppose $s_{xy} = u_{xz}$. 
Then $v_{xy}=v_{xz}$. 
Since $s\in V(T_{w_{y}}) \subseteq V(T_{s_{xy}})$ and $s_{xy} = u_{xz}$, we deduce that $s\in V(T_{u_{xz}})$ and $t\in \OUT(v_{xz})$ in this case. 
However, $V(T_{w_{y}})$ does not see $\OUT(s_{xy}) \setminus \{v_{xy}\}$ (by property~\ref{prop:does_not_see} of~\ref{lem:attaching_to_special_path}), and $t\in \OUT(v_{xz}) \subseteq \OUT(u_{xz}) \setminus \{v_{xz}\} = \OUT(s_{xy}) \setminus \{v_{xy}\}$, a contradiction. 

Next, assume that $s_{xy} \neq u_{xz}$. 
Then $s_{xy} \notin V(T_{u_{xz}})$, because the parent $v_{xy}$ of $s_{xy}$ is on the path $P_x$, and $P_x$ avoids $V(T_{u_{xz}})$. 
Since $s_{xy} \notin V(T_{s_{xz}})$ and $s_{xy} \neq v_{xz}$, it follows that $s_{xy} \in \OUT(v_{xz})$.  
Since $s\in V(T_{w_{y}}) \subseteq V(T_{s_{xy}})$ and since $s_{xy}$ is not an ancestor of $v_{xz}$ (otherwise $V(T_{s_{xy}})$ would contain $v_{xz}$, which is on the path $P_x$), we deduce that $V(T_{s_{xy}}) \subseteq \OUT(v_{xz})$, and thus $s \in \OUT(v_{xz})$.  
It then follows that $t\in V(T_{u_{xz}})$. 
Observe that $u_{xz}$ is neither an ancestor of $v_{xy}$ (otherwise $V(T_{u_{xz}})$ would contain $v_{xy}$, which is on the path $P_x$) nor a descendant of $s_{xy}$ (otherwise $V(T_{s_{xy}})$ would contain $v_{xz}$ since $u_{xz} \neq s_{xy}$, which is a vertex of $P_x$). 
Hence, we deduce that $V(T_{u_{xz}}) \subseteq \OUT(s_{xy}) \setminus \{v_{xy}\}$. 
However, the edge $st$ then contradicts the fact that $V(T_{w_{y}})$ does not see $\OUT(s_{xy}) \setminus \{v_{xy}\}$ (c.f.\ property~\ref{prop:does_not_see} of~\ref{lem:attaching_to_special_path}). 
Therefore, $V(Q_z) \cap V(T_{w_y}) = \emptyset$, as claimed. 
Property~\ref{prop:sibling} follows.

We now verify property~\ref{prop:new}. 
First, we show~\ref{prop:new} holds for the leaf $y$ of $T^1$. 
Note that $V(T_{w_y}) \subset V(T_{s_{xy}}) \subset V(T_{w_x})$.
Thus, $V(T_{w_y})$ and $\bigcup_{p\in V(T^1) \setminus \set{x,y,z}} V(Q_p)$ are disjoint by induction and property~\ref{prop:new} for the leaf $x$ of $T^1-\set{y,z}$.
Since $V(T_{w_y}) \subset V(T_{s_{xy}})$ and $V(T_{s_{xy}}) \cap V(Q_x)=\emptyset$ (by property~\ref{prop:Qsi} of~\ref{lem:attaching_to_special_path}), we deduce that $V(T_{w_y}) \cap V(Q_x)=\emptyset$.
Moreover, $V(T_{w_y}) \cap V(Q_z)=\emptyset$, by property~\ref{prop:sibling} shown above. 
This proves property~\ref{prop:new} for the leaf $y$ of $T^1$, and also for the leaf $z$ by symmetry. 

Every other leaf $q$ of $T^1$ is also a leaf in $T^1 - \set{y,z}$.
By induction, $V(T_{w_q})$ and $\bigcup_{p\in V(T^1) \setminus \set{q,y,z}} V(Q_p)$ are disjoint.
Moreover, $V(T_{v_q})$ and $V(T_{v_x})$ are disjoint, by property~\ref{prop:kind-of-obvious}.
Since $V(Q_y)$ and $V(Q_z)$ are contained in $V(T_{v_x})$ (by property~\ref{prop:Qax} of~\ref{lem:attaching_to_special_path}) and $V(T_{w_q})\subset V(T_{v_q})$, it follows that $V(T_{w_q})$ and $V(Q_y) \cup V(Q_z)$ are also disjoint.
Property~\ref{prop:new} follows.

To conclude the proof, it only remains to verify that $(T^1, \{(P_p, Q_p):p \in V(T^1)\})$ is a binary pear tree in $G$, and that it is clean. 
Recall that $(T^1-\{y,z\}, \{(P_p, Q_p):p \in V(T^1 - \{y,z\})\})$ is a binary pear tree, by induction. 
By construction, $P_y \subseteq Q_y$ and $P_z \subseteq Q_z$, $P_y$ and $P_z$ each have length at least $2$, and both are $P_x$-ears. 
Clearly, property~\ref{def:bpt_parent} of the definition of binary pear trees holds. 
Property~\ref{def:bpt_no_sibling} holds vacuously, since $T^1$ is a full binary tree, and thus every non-root vertex of $T^1$ has a sibling. 
Hence, it only remains to show that property~\ref{def:bpt_sibling} holds.  

Let $p$ be a non-root vertex of $T^1$, and let $p'$ denote its sibling. 
First we want to show that no internal vertex of $Q_p$ is in $\bigcup_{q \in V(T^1) \setminus (V(T^1_p) \cup V(T^1_{p'}))} V(Q_q)$. 

If $p$ is an ancestor of $x$ in $T^1$ (including $x$) then this holds thanks to property~\ref{def:bpt_sibling} of the binary pear tree $(T^1-\{y,z\}, \{(P_q, Q_q):q \in V(T^1 - \{y,z\})\})$. 

Next, suppose $p$ is not an ancestor of $x$ in $T^1$ and $p$ is not $y$ nor $z$. 
Then we already know that no internal vertex of $Q_p$ is in $\bigcup_{q \in V(T^1-\{y,z\}) \setminus (V(T^1_p) \cup V(T^1_{p'}))} V(Q_q)$, again by property~\ref{def:bpt_sibling} of the binary pear tree $(T^1-\{y,z\}, \{(P_q, Q_q):q \in V(T^1 - \{y,z\})\})$. 
Thus it only remains to show that if some internal vertex of $Q_p$ is in $Q_y$ then $y$ is a descendant of $p$ or of $p'$, and that the same holds for $Q_z$. 
By symmetry, it is enough to prove this for $Q_y$.  
So let us assume that some internal vertex of $Q_p$ is in $Q_y$. 
Note that $V(Q_y) \subseteq V(T_{w_x})\cup\set{v_x}$, by property~\ref{prop:Qax} of~\ref{lem:attaching_to_special_path}. 
By property~\ref{prop:new} of the inductive statement, $V(T_{w_x})$ is disjoint from $V(Q_p)$.
Thus, the only vertex that the paths $Q_p$ and $Q_y$ can have in common is $v_x$. 
Since $v_x$ is an internal vertex of $Q_p$ (by our assumption) and since $v_x \in V(Q_x)$, from property~\ref{def:bpt_sibling} of the binary pear tree $(T^1-\{y,z\}, \{(P_q, Q_q):q \in V(T^1 - \{y,z\})\})$ we deduce that $x$ is a descendant of $p$ or $p'$, and hence so is $y$, as desired. 

Finally, consider the case where $p$ is $y$ or $z$, say $y$. 
Recall that $V(Q_y)\subseteq V(T_{w_x})\cup\set{v_x}$. 
Note also that $v_x$ cannot be an internal vertex of $Q_y$, since $v_x\in V(P_x)$ and $Q_y$ is a $P_x$-ear. 
Hence, all internal vertices of $Q_y$ are in $V(T_{w_x})$.  
Since $V(T_{w_x})$ and $V(Q_q)$ are disjoint for all $q\in V(T^1)\setminus\set{x,y,z}$ (by induction, using property~\ref{prop:new} on the leaf $x$ of $T^1-\{y,z\}$). 
Thus, it only remains to show that no internal vertex of $Q_y$ is in $Q_x$. 
This is the case, because $Q_y$ is a $P_x$-ear, and 
$V(Q_x) \setminus V(P_x) \subseteq \OUT(w_x) \setminus \{v_x\}$ (by property~\ref{prop:QP} of~\ref{lem:attaching_to_special_path}). 

To establish property~\ref{def:bpt_sibling}, it remains to show that no internal vertex of $P_p$ is in $Q_{p'}$, for every two siblings $p, p'$ of $T^1$. 
If $\{p, p'\} \neq \{y, z\}$, this is true by property~\ref{def:bpt_sibling} of the binary pear tree $(T^1-\{y,z\}, \{(P_q, Q_q):q \in V(T^1 - \{y,z\})\})$. 
Thus by symmetry, it is enough to show that no internal vertex of $P_y$ is in $Q_{z}$. 
This holds because all internal vertices of $P_y$ are in $V(T_{w_y})$ (since $P_y$ is a $(v_{y}, w_{y}, v'_{y})$-special path) and $V(Q_z) \cap V(T_{w_y}) = \emptyset$ by~\ref{prop:sibling}. 

This concludes the proof that $(T^1, \{(P_p, Q_p):p \in V(T^1)\})$ is a binary pear tree. 
Finally, note that it is clean because the binary pear tree $(T^1-\{y,z\}, \{(P_q, Q_q):q \in V(T^1 - \{y,z\})\})$ is clean (by induction), and the end $v'_x$ of $P_x$ is not in $Q_{y}$, since $V(Q_{y}) \subseteq V(T_{w_x}) \cup \{v_x\}$ (by property~\ref{prop:Qax} of~\ref{lem:attaching_to_special_path}), and since $v'_x \notin V(T_{w_x}) \cup \{v_x\}$, and similarly $v'_x$ is not in $Q_z$ either. 
\end{proof}

\section{Proof of Main Theorems} \label{sec:theproof}

We have the following quantitative version of \ref{thm:main-CBT}.
 
\begin{theorem}\label{thm:main-CBT-quantitative}
For all integers $\ell \geq 1$ and $k\geq 9\ell^2 - 3\ell +1$, every $2$-connected graph $G$ with a $\Gamma_k$ minor contains $\Gamma_{\ell}^{+}$ or $\nabla_\ell$ as a minor.
\end{theorem}

\begin{proof}
Among all $2$-connected graphs containing $\Gamma_k$ as a minor, but containing neither $\Gamma_{\ell}^{+}$ nor $\nabla_\ell$ as a minor, choose $G$ with $|E(G)|$ minimum.   
Since $\Gamma_k$ has maximum degree $3$, $G$ contains a subdivision of $\Gamma_k$.  Therefore, $G$ is a minor-minimal $2$-connected graph containing a subdivision of $\Gamma_k$.
By~\ref{lem:findingbeds},  $G$ has a clean binary pear tree  $(T^1, \mathcal{B})$, with $T^1 \simeq \Gamma_{3\ell-2}$.  By~\ref{thm:pears make ears}, $G$ has a minor $H$ such that $H$ has a clean binary ear tree  $(T^1, \mathcal{P})$, with $T^1 \simeq \Gamma_{3\ell-2}$. By~\ref{thm:binaryminors}, $H$ contains  $\Gamma_{\ell}^{+}$ or $\nabla_\ell$ as a minor, and hence so does $G$. 
\end{proof}

We have the following quantitative version of \ref{thm:main}.

\begin{theorem}\label{thm:main-quantitative}
For every integer $\ell\geq 1$, every $2$-connected graph $G$ of pathwidth at least $2^{9\ell^2 - 3\ell +2}-2$ contains $\Gamma_{\ell}^{+}$ or  $\nabla_\ell$ as a minor.
\end{theorem}

\begin{proof}
As mentioned in Section~\ref{sec:intro}, Bienstock~et~al.~\cite{QuicklyForest} proved that for every forest $F$, every graph with pathwidth at least $|V(F)|-1$ contains $F$ as a minor. Let $k:= 9\ell^2 - 3\ell +1$. Note that $|V(\Gamma_k)|=2^{k+1}-1$. By assumption, $G$ has pathwidth at least $2^{k+1}-2$. Thus $G$ contains $\Gamma_k$ as a minor. The result follows from \ref{thm:main-CBT-quantitative}. 
\end{proof}

Finally, we have the following quantitative version of \ref{SeymourConj}.

\begin{theorem}
For every apex-forest $H_1$ and outerplanar graph $H_2$, if $\ell:=\max\{|V(H_1)|,|V(H_2)|,2\}-1$ then every 2-connected graph $G$ of pathwidth at least $2^{9\ell^2 - 3\ell +2}-2$ contains $H_1$ or $H_2$ as a minor. 
\end{theorem}

\begin{proof}
By \ref{thm:main-quantitative}, $G$ contains $\Gamma_{\ell}^{+}$ or  $\nabla_\ell$ as a minor. 
In the first case, by \ref{CBTuniversal},  $H_1$ is a minor of $\Gamma_{|V(H_1)|-1}^{+}$ and thus of $G$ (since $\ell\geq |V(H_1)|-1$). 
In the second case, by \ref{OuterplanarUniversal}, $H_2$ is a minor of $\nabla_{|V(H_2)|-1}$ and thus of $G$  (since $\ell\geq |V(H_2)|-1$). 
\end{proof}

\textbf{Acknowledgements.} We would like to thank Robin Thomas for informing us of the PhD thesis of Thanh N.\ Dang~\cite{Dang17}. 
We are also grateful to an anonymous referee for several helpful remarks and suggestions. 

\bibliographystyle{myNatbibStyle}
\bibliography{bibliography} 
\end{document}